\documentclass[12pt,twoside]{amsart}
\usepackage{amssymb,amsmath,amscd,enumerate,verbatim,xcolor,mathtools,fullpage}
\usepackage[colorlinks=true,linkcolor=blue,citecolor=blue]{hyperref}
\usepackage[all]{xy}
\usepackage{cleveref}
\usepackage{graphicx}
\usepackage{tikz-cd}
\usepackage{multirow,tabularx}
\usepackage[numbers,sort&compress]{natbib}

\numberwithin{equation}{section}

\newtheorem{thm}{\bf Theorem}[section]
\newtheorem{lem}[thm]{\bf Lemma}

\newtheorem{cor}[thm]{\bf Corollary}
\newtheorem{prop}[thm]{\bf Proposition}
\newtheorem{prob}[thm]{\bf Problem}

\newtheorem{claim}{Claim}[thm]

\crefname{thm}{theorem}{theorems}
\crefname{lem}{lemma}{lemmas}
\crefname{prop}{proposition}{propositions}

\theoremstyle{definition}
\newtheorem{defn}[thm]{Definition}
\newtheorem{ex}[thm]{Example}
\newtheorem{notn}[thm]{Notation}
\newtheorem{rem}[thm]{Remark}

\newtheorem*{acknowledgment}{Acknowledgments}

\newtheorem*{thm*}{Theorem}

\DeclareMathOperator{\depth}{depth}

\DeclareMathOperator{\IN}{IN}
\DeclareMathOperator{\Inc}{Inc}
\DeclareMathOperator{\ind}{ind}

\DeclareMathOperator{\inm}{im}

\DeclareMathOperator{\msupp}{msupp}
\DeclareMathOperator{\Msupp}{Msupp}

\DeclareMathOperator{\pd}{pd}

\DeclareMathOperator{\reg}{reg}
\DeclareMathOperator{\spi}{sp}

\DeclareMathOperator{\Supp}{supp}

\DeclareMathOperator{\Tor}{Tor}

\newcommand{\Z}{{\mathbb Z}}
\newcommand{\N}{{\mathbb N}}

\newcommand{\R}{{\mathbb R}}

 \newcommand{\Fc}{{\mathcal F}}
\def\Gc{{\mathcal G}}

\def\Icc{{\mathcal I}}
\def\Jcc{{\mathcal J}}

\def\mm{{\mathfrak m}}

\newcommand\bsa{{\boldsymbol a}}
\newcommand\bsb{{\boldsymbol b}}
\newcommand\bse{{\boldsymbol e}}

\newcommand{\bfo}{\mathbf{0}}

\newcommand{\wti}{\widetilde}

\newcommand{\kk}{\Bbbk}

\newcommand\defas{\coloneqq}

\begin{document}

\title[Asymptotic depth of invariant chains]{Asymptotic depth of invariant chains of edge ideals}
\author[T.Q. Hoa]{Tran Quang Hoa}
\address{University of Education, Hue University, 34 Le Loi St., Hue City, Viet Nam}
\email{tranquanghoa@hueuni.edu.vn}

\author[D.T. Hoang]{Do Trong Hoang}
\address{Faculty of Mathematics and Informatics, Hanoi University of Science and Technology, 1 Dai Co Viet, Hai Ba Trung, Hanoi, Vietnam}
\email{hoang.dotrong@hust.edu.vn}

\author[D.V. Le]{Dinh Van Le}
\address{Department of Mathematics, FPT University, Hanoi, Vietnam}
\email{dinhlv2@fe.edu.vn}

\author[H.D. Nguyen]{Hop D. Nguyen}
\address{Institute of Mathematics, Vietnam Academy of Science and Technology, 18 Hoang Quoc Viet, 10307 Hanoi, Vietnam}
\email{ngdhop@gmail.com}

\author[T.T. Nguyen]{Th\'ai Th\`anh Nguy$\tilde{\text{\^E}}$n}
\address{McMaster University, Department of Mathematics and Statistics, 1280 Main Street West, Hamilton, Ontario, Canada}
\address{University of Education, Hue University, 34 Le Loi St., Hue, Viet Nam}
\email{nguyt161@mcmaster.ca}
\email{tnguyen11@tulane.edu}


\begin{abstract}
We completely determine the asymptotic depth,  equivalently, the asymptotic projective dimension of a chain of edge ideals that is invariant under the action of the monoid $\Inc$ of increasing functions on the positive integers. Our results and their proofs also reveal  surprising combinatorial and topological properties of corresponding graphs and their independence complexes. In particular, we are able to determine the asymptotic behavior of all reduced homology groups of these independence complexes.
\end{abstract}


\maketitle

\section{Introduction}

For $n\ge 1$ let $R_n=\kk[x_1,\dots,x_n]$ be the polynomial ring in $n$ variables over a field $\kk$. A vibrant area of research at the crossroads of algebraic geometry, combinatorics, commutative algebra, group theory, representation theory, and statistics concerning chains of ideals $\Icc=(I_n)_{n\ge 1}$ with $I_n\subseteq R_n$ for $n\ge 1$ that are invariant under the actions of the infinite general linear group, the infinite symmetric group, or more generally, the monoid $\Inc$ of strictly increasing functions; see, e.g. \cite{AH07,CEF15,Co67,Co87,Dr10,Dr14,DEF,DEKL,DK14,HS12,LR20,LR24,NR17,NR19,SS16,SS17}. See also \cite{KLR22,L24,LR21,LC,LD} for related directions in discrete geometry, convex optimization, and machine learning.

It follows from a celebrated result of Cohen \cite{Co67} that any $\Inc$-invariant chain $\Icc=(I_n)_{n\ge 1}$ \emph{stabilizes} (see \Cref{sec.InvChain} for more details). That is, from some index $r$ on, any ideal $I_n$ with $n\ge r$ is completely determined by $I_r$ up to the $\Inc$-action. One might therefore hope that invariants related to these ideals are well-behaved. In \cite{LNNR1,LNNR2}, it is conjectured that the regularity and projective dimension of $I_n$ are eventually linear functions in $n$. See \cite{Na} for extensions of these conjectures to OI-modules. Although significant evidence for the conjectures has been obtained \cite{Ga24,LN22,Mu,MR22,Ra21,SS22}, they remain wide open.

Recently, the regularity conjecture \cite{LNNR1} has been verified for chains of edge ideals in \cite{HNT2024}. It is shown that the regularity of an ideal in such a chain is eventually a constant, that, somewhat surprisingly, can only be 2 or 3. Proving this result requires a deep understanding of the chain of graphs corresponding to the chain of edge ideals. It turns out, for instance, that eventually such a graph can only have an induced matching number of at most 2.

As a continuation of \cite{HNT2024}, this paper explores further properties of $\Inc$-invariant chains of edge ideals. Our main goal is to verify the projective dimension conjecture \cite{LNNR2} for those chains. More specifically, we are interested in the following stronger problem.

\begin{prob}
\label{prob.asymptoticdepth}
Let $\Icc=(I_n)_{n\ge 1}$ be an $\Inc$-invariant chain of edge ideals. Compute explicitly  $\depth(R_n/I_n)$ for large $n$, in combinatorial terms.
\end{prob}

The main result of the paper resolves this problem, providing an explicit and simple formula for $\depth(R_n/I_n)$ when $n\gg0$. In particular, it shows that $\depth(R_n/I_n)$ is eventually a constant. This implies, via the Auslander--Buchsbaum formula, that the projective dimension of $I_n$ is eventually a linear function in $n$. Below, the \emph{support} of a monomial ideal is the set of variables that divide at least one minimal monomial generator of that ideal.

\begin{thm}
\label{thm_limdepth_main}
 Let $\Icc=(I_n)_{n\ge 1}$ be an $\Inc$-invariant chain of \textup{(}eventually non-zero\textup{)} edge ideals with the stability index $\ind(\Icc)=r$. Assume that $x_1$ and $x_r$ belong to the support of $I_r$. Denote $j_q=\max\{j\mid  1\le j\le r, x_1x_j\in I_r \}$ and $ \spi(\Icc)=\min\{j-i \mid 1\le i\le j\le r, x_ix_j\in I_r\}$. Then the following hold for all $n\ge 3r$.
 \begin{enumerate}
 \item  If $j_q= r$, then $\depth (R_n/I_n) =\spi(\Icc).$
  
\item If $j_q< r$, then $\depth (R_n/I_n) =\min\{\spi(\Icc), 2\}.$
\end{enumerate}
\end{thm}

It should be noted that the assumption that $x_1$ and $x_r$ belong to the support of $I_r$ in the previous theorem is not restrictive, as the general case can always be reduced to this case by simple index shifts (see \Cref{lem.normalizingindexshift}). Note also that the index $j_q$ and the invariant $\spi(\Icc)$, which we call the \emph{sparsity index} of the chain $\Icc$, play a crucial role in this result.

\Cref{thm_limdepth_main} exhibits a dichotomy for $\depth(R_n/I_n)$ when $n\gg0$. This is very similar to the dichotomy for $\reg(I_n)$ mentioned above. It would be interesting to have some explanation for this similarity.

Our proof of \Cref{thm_limdepth_main} makes use of Takayama's formula, which allows us to interpret the problem of computing $\depth(R_n/I_n)$ as the problem of computing certain reduced homology groups of the independence complex $\IN(G_n)$ of the graph $G_n$ corresponding to the ideal $I_n$. This approach, therefore, inspires the following problem.

\begin{prob}
\label{pb.homology}
 Determine all reduced homology groups of $\IN(G_n)$ for $n\gg0$.
\end{prob}

In the present paper, we also obtain a complete solution to this problem. As in \Cref{thm_limdepth_main}, the index $j_q$ as well as the sparsity index $\spi(\Icc)$ are essential for the statement of our result.

\begin{thm}
\label{thm_limhomology_intro}
Keep the assumptions of \Cref{thm_limdepth_main}. Denote by $\wti{H}_i(\IN(G_n))$ the $i$-th reduced homology group of $\IN(G_n)$ over the field $\kk$. Then there exist two nonnegative integers $\alpha, \beta$ depending only on the chain $\Icc$ such that the following hold for all $n\gg0$.
\begin{enumerate}
    \item 
    $\wti{H}_i(\IN(G_n))=0$ for $i\ne 0,1$.
    \item 
    $\wti{H}_0(\IN(G_n))
    \cong
    \begin{cases}
     \kk^{n-\alpha}&\text{if }\ \spi(\Icc)=1,\\
     0&\text{if }\ \spi(\Icc)\ge2.
    \end{cases}$
     \item 
     $\wti{H}_1(\IN(G_n))\cong\kk^\beta.$ Furthermore, if $\spi(\Icc)\ge2$, then 
$\beta=\begin{cases}
0&\text{if }\ j_q=r,\\
1&\text{if }\ j_q<r.
\end{cases}$
\end{enumerate}
\end{thm}

\Cref{thm_limhomology_intro} is an unexpected outcome of our approach to \Cref{prob.asymptoticdepth}. On the other hand, the proof of Item (ii)  (which is the harder part) of \Cref{thm_limdepth_main} depends crucially on \Cref{thm_limhomology_intro}(iii) (see \Cref{prop.H1nonvanishing}). It is worth mentioning that in the situation where $\spi(\Icc)\ge 2$ and $j_q<r$, \Cref{thm_limhomology_intro} provides the remarkable information that $\IN(G_n)$ has the same reduced homology as the sphere $\mathbb{S}^1$ for all $n\gg 0$.

This paper marks the first successful application of Takayama's formula to problems on  $\Inc$-invariant chains of monomial ideals. We believe that many of our results and proofs, for example those in \Cref{sec.Bound} and parts of \Cref{sec.MaxDepth,sec.MinDepth}, are amenable to more general situations beyond edge ideals of graphs. Moreover, as an extention of \Cref{pb.homology}, the theme of asymptotic homology stability of simplicial complexes associated to $\Inc$-invariant chains of monomial ideals is worthy of further investigation.

Let us now describe the structure of the paper. \Cref{sec.prelim} provides graph terminology and some basic properties of monomial ideals. In \Cref{sec.InvChain}, we review $\Inc$-invariant chains of edge ideals and present some auxiliary results that reveal interesting structures of corresponding graphs. Upper and lower bounds for the asymptotic depth of an $\Inc$-invariant chain of edge ideals are given in \Cref{sec.Bound}. The next two sections show that the asymptotic depth of any $\Inc$-invariant chain of edge ideals always attains one of the two bounds established in \Cref{sec.Bound}, thus proving \Cref{thm_limdepth_main}. \Cref{sec.homology} is devoted to the proof of \Cref{thm_limhomology_intro}. Lastly, the Appendix proves a technical result from \Cref{sec.MinDepth} that is decisive to the proof of \Cref{thm_limdepth_main}(ii).


\section{Preliminaries}
\label{sec.prelim}

We collect here necessary notions and results concerning graphs and monomial ideals. For unexplained terminology, the reader is referred to \cite{HH,Vi}. Throughout the section, let $S=\kk[x_1,\ldots,x_n]$ be a standard graded polynomial ring over a field $\kk$ and $\mm=\langle x_1,\ldots,x_n\rangle$ its graded maximal ideal.

\subsection{Depth, projective dimension, and regularity}

Let $M$ be a finitely generated graded $S$-module. The \emph{projective dimension} and the \emph{(Castelnuovo-Mumford) regularity} of $M$ are defined as 
\begin{align*}
    \pd(M)&:=\max\{i\mid \Tor^S_i(M,\kk) \ne 0\},\\
    \reg(M)&:=\max\{j-i\mid \Tor^S_i(M,\kk)_j\ne 0\}.
\end{align*}
In particular, for any nonzero homogeneous ideal $I\subseteq S$ we have
\[
\pd(S/I)=\pd(I)+1
\quad\text{and}\quad
\reg(S/I)=\reg(I)-1.
\]
The \emph{depth} of $M$ is related to its projective dimension via the Auslander--Buchsbaum formula: 
\[
\pd(M)+\depth(M)=n.
\]
In this paper, we will mainly use the following interpretations of depth and regularity in terms of local cohomology modules
\begin{align*}
\depth(M)&=\min\{i\mid H^i_\mm(M)\ne 0\},\\
\reg(M)&=\max\{i+j\mid H^i_\mm(M)_j\ne 0\},
\end{align*}
where $H^i_\mm(M)$ denotes the $i$th local cohomology module of $M$ with respect to $\mm$.


\subsection{Takayama's formula}
\label{subsec.Takayama}

For any monomial ideal $I\subseteq S$, there is a $\Z_{\ge0}^n$-grading on $S/I$ that is inherited from the natural $\Z_{\ge0}^n$-grading of $S$. This induces a $\Z^n$-grading on the local cohomology module $H^i_\mm(S/I)$. The dimensions of the $\Z^n$-graded components of $H^i_\mm(S/I)$ are described by Takayama's formula, which we now recall. For brevity, we write $[n]=\{1,\dots,n\}.$ 
Let $\bsa =(a_1,\ldots,a_n) \in \Z^n$ be an integral vector. Set $x^{\bsa}=x_1^{a_1}\cdots x_n^{a_n}$. Also, denote
\[
\Supp(\bsa)=\{i\in [n] \mid a_i\ne0\}
\quad\text{and}\quad
G_\bsa=\{i\in [n] \mid a_i<0\}.
\]
For $F\subseteq[n]$, let $S_{(F)}=S[x_i^{-1}\mid i\in F]$. The \emph{degree complex of $I$ with respect to $\bsa$} is defined as 
\[
\Delta_\bsa(I)=\{F\setminus G_\bsa \mid G_\bsa \subseteq F\subseteq [n], x^{\bsa} \notin IS_{(F)}\}.
\]
Then Takayama's formula \cite[Theorem 1]{Ta} (see also \cite[Section 1]{MTr11}) states that
\[
\dim_\kk H^i_\mm(S/I)_\bsa = \dim_\kk \wti{H}_{i-|G_\bsa|-1}(\Delta_\bsa(I)) \quad \text{for all $i\in \Z$ and all $\bsa\in \Z^n$},
\]
where $\wti{H}_{i}(\Delta_\bsa(I))$ denotes the $i$-th reduced homology of $\Delta_\bsa(I)$ over $\kk$.

Takayama's formula is useful for studying $\depth(S/I)$. 
We gather here two simple results in this direction. For $F\subseteq[n]$, let $S_F=\kk[x_i \mid i\notin F]$ and $I_{F}=IS_{(F)} \cap S_F$. Then $I_F$ is the monomial ideal in $S_F$ obtained from $I$ by setting $x_i=1$ for all $i\in F$. When $F=\{j\}$ for some $j\in[n]$, we simply write $S_j$ and $I_j$ instead of $S_{\{j\}}$ and $I_{\{j\}}$, respectively.

The following criterion for the vanishing of $H^1_\mm(S/I)$ from \cite[Proposition 1.6]{TTr14} can be used to give a lower bound for $\depth(S/I)$.

\begin{prop}
\label{prop_1stLocalCohom}
Let $I\subseteq S$ be a monomial ideal. The following are equivalent:
\begin{enumerate}[\quad \rm (i)]
 \item 
 $H^1_\mm(S/I)=0$;
 \item 
 $\Delta_\bsa(I)$ is connected for all $\bsa \in \Z_{\ge0}^n$ and $\depth (S_j/I_j)\ge 1$ for all $j\in[n]$.
\end{enumerate}
\end{prop}

The next result provides an upper bound for $\depth(S/I)$.

\begin{prop}
\label{prop_SmallDepthCrit}
Let $I\subseteq S$ be a monomial ideal. Then
\[
\depth(S/I)\le\min\{|F|\mid F\subseteq [n], \depth(S_{F}/I_{F})=0\}.
\]
\end{prop}

In order to prove this proposition, we need the following lemma, which is an easy consequence of \cite[Corollary 1.4]{TTr14}.

\begin{lem}
\label{lem_TrivialComplexCrit}
Let $I\subseteq S$ be a monomial ideal and $F$ a subset of $[n]$. For any $\bsa \in \Z^n$, let $\bsa_+$ be the vector obtained from $\bsa$ by replacing each negative entry with zero. Then the following are equivalent:
\begin{enumerate}
 \item 
 $\wti{H}_{-1}(\Delta_\bsa(I))=0$ for all $\bsa \in \Z^n$ with $G_\bsa=F$;
  \item 
  $\Delta_\bsa(I)\neq \{\emptyset\}$ for all $\bsa \in \Z^n$ with $G_\bsa=F$;
  \item 
  $\depth(S_{F}/I_{F})\ge 1$.
\end{enumerate}
\end{lem}
\begin{proof}

 (i) $\Leftrightarrow$ (ii) is clear, since $\wti{H}_{-1}(\Delta_\bsa(I))=0$ if and only if $\Delta_\bsa(I)\neq \{\emptyset\}$.
 
 (iii) $\Rightarrow$ (ii): Let $\wti{I_{F}}=\bigcup_{k\ge1}(I_{F}:\mm_F^k)$
be the \emph{saturation} of $I_{F}$, where $\mm_F=\langle x_i \mid i\notin F\rangle$ denotes the graded maximal ideal of $S_F$. Then it is well-known that $H^0_\mm(S_F/I_F)=\wti{I_{F}}/I_F$.
Since $\depth(S_{F}/I_{F})\ge 1$, we have $H^0_\mm(S_F/I_F)=0$, and hence $\wti{I_{F}}=I_{F}$. It follows that $x^{\bsa_+}\not\in\wti{I_{F}}\setminus I_{F}$ for all $\bsa \in \Z^n$ with $G_\bsa=F$. By \cite[Corollary 1.4]{TTr14}, this means that $\Delta_\bsa(I)\neq \{\emptyset\}$ for all such $\bsa$.
 
 (ii) $\Rightarrow$ (iii): Assume, to the contrary, that $\depth(S_{F}/I_{F})= 0$. Then $H^0_\mm(S_F/I_F)=\wti{I_{F}}/I_F\ne 0$. So there exists $\bsb\in \Z_{\ge0}^n$ such that $\Supp(\bsb) \subseteq [n]\setminus F$ and $x^\bsb\in \wti{I_{F}}\setminus I_{F}$. Let $\bse_1,\ldots,\bse_n$ be the standard unit vectors of $\Z^n$. Setting $\bsa=\bsb-\sum_{i\in F}\bse_i$, we get $G_\bsa=F$ and $\bsa_+=\bsb$. Hence $x^{\bsa_+}\in \wti{I_{F}}\setminus I_{F}$, which by \cite[Corollary 1.4]{TTr14} implies that  $\Delta_\bsa(I)=\{\emptyset\}$. This contradiction shows that $\depth(S_{F}/I_{F})\ge 1$, as desired.
\end{proof}

Let us now prove \Cref{prop_SmallDepthCrit}.

\begin{proof}[Proof of \Cref{prop_SmallDepthCrit}]
It suffices to show that if there exists $F\subseteq [n]$ with $|F|=i$ and
$
\depth(S_{F}/I_{F})=0
$,
then $H^{i}_\mm(S/I)\ne0.$ Equivalently, this amounts to showing that if
$H^i_\mm(S/I)=0$, then for any $F\subseteq [n]$ with $|F|=i$ one has
$
\depth(S_{F}/I_{F})\ge 1.
$
Indeed, let $\bsa\in \Z^n$ be any vector with $G_\bsa=F$. Then by Takayama's formula
\[
\dim_\kk \wti{H}_{-1}(\Delta_\bsa(I))= \dim_\kk H^i_\mm(S/I)_\bsa 
=0.
\]
 Hence $\wti{H}_{-1}(\Delta_\bsa(I))=0$, and therefore  
 $
\depth(S_{F}/I_{F})\ge 1
$
by virtue of \Cref{lem_TrivialComplexCrit}. 
\end{proof}


\subsection{Graphs and edge ideals}
\label{subsec.IndCom}

Let $G$ be a simple graph with vertex set $V(G)$ and edge set $E(G)$. 
A \emph{subgraph} of $G$ is a graph $H$ with $V(H)\subseteq V(G)$ and $E(H)\subseteq E(G)$. If, moreover, $E(H)= E(G)\cap \binom{V(H)}{2}$, then $H$ is called an \emph{induced subgraph} of $G$.
The \emph{complement} $G^c$ of $G$ is the graph on $V(G)$ with edge set $\binom{V(G)}{2}\setminus E(G)$. For an integer $m\ge 3$, a \emph{cycle} of length $m$ is a graph $C_m$ with $V(C_m)=\{v_1,\ldots,v_m\}$ and
$
E(C_m)=\left\{\{v_1,v_2\}, \{v_2,v_3\},\ldots,\{v_m,v_1\} \right\}.
$
The graph $G$ is called \emph{weakly chordal} (or \emph{weakly triangulated}) if neither $G$ nor $G^c$ contains an induced cycle of length at least 5. 

A subset $U\subseteq V(G)$ is called {\it independent} if the vertices in $U$ are pairwise non-adjacent. The \emph{independence complex} of $G$, denote by $\IN(G)$, is the simplicial complex
whose faces are the independent sets of $G$. It is evident that the 1-skeleton of $\IN(G)$ is exactly $G^c$. This yields the following fact that should be well-known.

\begin{lem}
\label{lem.vanishingofH0}
Let $G$ be a graph with at least one vertex. Then  $\wti{H}_0(\IN(G))\cong \kk^{\mathfrak{c}(G^c)-1}$, where $\mathfrak{c}(G^c)$ denotes the number of connected components of $G^c$. In particular, $\wti{H}_0(\IN(G))=0$ if and only if $G^c$ is connected.
\end{lem}

For a vertex $v\in V(G)$, its \emph{open neighborhood} $N(v)$ is the set of vertices $u\neq v$ that are adjacent to $v$, and its \emph{closed neighborhood} is $N[v]\defas N(v) \cup \{v\}$. More generally, for a subset $U\subseteq V(G)$, we define $N[U]\defas\bigcup_{v\in U}N[v].$ Let $G\setminus U$ denote the graph obtained from $G$ by deleting all vertices in $U$ and all edges adjacent to those vertices. Observe that $G\setminus U$ is an induced subgraph of $G$.
The following result, which is essentially a consequence of the Mayer--Vietoris long exact sequence, can be found in \cite[Theorem 3.5.1]{Jon11} or \cite[Section 2.1]{Kim22}.

\begin{lem}
\label{lem.MayerVietoris}
Let $G$ be a graph and $v$ a vertex of $G$. Then there is a long exact sequence
\begin{align*}
\cdots  \to \wti{H}_i(\IN(G\setminus N[v])) \to  \wti{H}_i(\IN(G\setminus v)) \to  \wti{H}_i(\IN(G)) \to \wti{H}_{i-1}(\IN(G\setminus N[v])) \to \cdots.
\end{align*}
\end{lem}

From now on, assume that $V(G)=[n]$.
The \emph{edge ideal} of $G$ is defined as
\[
I(G)=\langle x_ix_j \mid \{i,j\} \in E(G)\rangle
\subseteq S.
\]
Evidently, $I(G)$ is the Stanley--Reisner ideal of the independence complex $\IN(G)$. Thus, in particular, $\dim(S/I(G))= \dim \IN(G)+1$.

Next, let us  recall some notions that are useful for studying the projective dimension of $I(G)$. A \emph{matching} in $G$ is a subset of $E(G)$ that consists of pairwise disjoint edges. If a matching forms the edge set of an induced subgraph of $G$, it is called an \emph{induced matching}. The \emph{induced matching number} $\inm(G)$ of $G$ is the largest cardinality of an induced matching in $G$.
A \emph{strongly disjoint family of complete bipartite subgraphs} of $G$, as introduced in \cite{Ki}, is a collection $\mathfrak{B}_1,\ldots,\mathfrak{B}_g$ of subgraphs of $G$ satisfying the following conditions:
\begin{enumerate}
\item each $\mathfrak{B}_i$ is a complete bipartite subgraph of $G$;
\item $V(\mathfrak{B}_i)\cap V(\mathfrak{B}_j)=\emptyset$ for all $1\le i<j\le g$;
\item there exists an induced matching $\{e_1,\ldots,e_g\}$ in $G$ with $e_i\in E(\mathfrak{B}_i)$ for $i=1,\ldots,g$.
\end{enumerate}
The following formula for $\pd(S/I(G))$ follows from \cite[Theorem 7.7]{NgV16}.

\begin{prop}
\label{prop.pdim.weaklychordal}
Let $G$ be a weakly chordal graph with at least one edge. Then 
\[
\pd(S/I(G))=\max\left\{\sum_{i=1}^g |V(\mathfrak{B}_i)|-g\right\},
\]
where the maximum is taken over all $1\le g\le \inm(G)$ and all strongly disjoint family of complete bipartite subgraphs $\mathfrak{B}_1,\ldots,\mathfrak{B}_g$ of $G$.
\end{prop}

We now describe a relationship between degree complexes of $I(G)$ and the independence complex $\IN(G)$, which, together with Takayama's formula, allows us to study $\depth(S/I(G))$ via $\IN(G)$. For a simplicial complex $\Delta$, let $\Fc(\Delta)$ denote the set of its facets.

\begin{lem}
\label{lem.FacetsofDegComplex}
Let $G$ be a graph on $[n]$ with edge ideal $I=I(G)$. Then for any $\bsa \in \Z^n$,
\[
\Fc(\Delta_\bsa(I)) =\{F \setminus G_\bsa \mid \Supp(\bsa)\subseteq F\subseteq [n], F\in \Fc(\IN(G))\}.
\]
In particular, the following hold.
\begin{enumerate}
\item 
$\Delta_\bfo(I(G))=\IN(G)$, where $\bfo$ denotes the zero vector of $\Z^n$.
\item 
If $\bsb\in \Z^n$ such that $\Supp(\bsb)=\Supp(\bsa)$ and $G_\bsb=G_\bsa$, then $\Delta_\bsb(I)=\Delta_\bsa(I)$.
\item 
If $\Supp(\bsa)\not\in\IN(G)$, then $\Delta_\bsa(I)=\emptyset$.  
\item 
If $G_\bsa\subsetneq\Supp(\bsa)$ and $\Supp(\bsa)\in \IN(G)$, then $\Delta_\bsa(I)$ is a cone.
\item 
$\Delta_\bsa(I)=\{\emptyset\}$ if and only if $\Supp(\bsa)=G_\bsa\in \Fc(\IN(G))$.
\end{enumerate}
\end{lem}

The proof of this result requires the following special case of \cite[Proposition 1.6]{NgTr19}.

\begin{prop}
\label{prop.Facets.General} 
Let $G$ be a graph on $[n]$ with edge ideal $I=I(G)$. For $\bsa\in \Z^n$, denote by $\bsa_+$ the vector obtained from $\bsa$ by setting every negative entry to zero. Let $F \subseteq [n]$ be such that $G_\bsa \subseteq F$. 
Then the following are equivalent:
\begin{enumerate}
\item 
$F \setminus G_\bsa\in\Fc(\Delta_\bsa(I))$;
\item 
$F \in \Fc(\IN(G))$ and $x^{\bsa_+} \not\in \mm_F$, where 
$\mm_F=\langle x_i\mid i\notin F\rangle$.
\end{enumerate}
\end{prop}

\begin{proof}[Proof of \Cref{lem.FacetsofDegComplex}]
From \Cref{prop.Facets.General} it follows that
\begin{align*}
F\setminus G_\bsa\in \Fc(\Delta_\bsa(I)) &\Longleftrightarrow F\in \Fc(\IN(G)) \quad \text{and} \quad x^{\bsa_+} \notin \mm_F \\
&\Longleftrightarrow F\in \Fc(\IN(G)) \quad \text{and} \quad \Supp(\bsa_+)\cap ([n]\setminus F)=\emptyset \\
& \Longleftrightarrow F\in \Fc(\IN(G)) \quad \text{and} \quad \Supp(\bsa_+) \subseteq F \\
& \Longleftrightarrow F\in \Fc(\IN(G)) \quad \text{and} \quad \Supp(\bsa) \subseteq F \quad \text{(as $G_\bsa\subseteq F$)}. 
\end{align*}
This proves the given description of $\Fc(\Delta_\bsa(I))$. The remaining assertions follow readily. 
\end{proof}

As a consequence of \Cref{prop_1stLocalCohom,prop.Facets.General}, we obtain the following.

\begin{cor}
\label{cor_depthge2crit}
Let $G$ be a graph on $[n]$ with edge ideal $I=I(G)$. Then the following are equivalent:
\begin{enumerate}
 \item $\depth(S/I)\ge 2$;
 \item $H^1_\mm(S/I)=0$;
 \item The complement $G^c$ of $G$ is connected.
\end{enumerate}
\end{cor}

\begin{proof}
Since $I$ is a squarefree non-maxiamal ideal, $H^0_\mm(S/I)=0$, and thus (i) $\Leftrightarrow$ (ii).

(ii) $\Rightarrow$ (iii): By \Cref{prop_1stLocalCohom} and  \Cref{lem.FacetsofDegComplex}, $\Delta_\bfo(I)=\IN(G)$ is connected. Since the 1-skeleton of $\IN(G)$ is precisely $G^c$, we deduce that $G^c$ is connected.

(iii) $\Rightarrow$ (ii): By \Cref{prop_1stLocalCohom}, we have to check the following:
\begin{enumerate}[\quad \rm (a)]
 \item $\Delta_\bsa(I)$ is connected for every $\bsa \in \Z^n_{\ge 0}$;
 \item $\depth(S_j/I_j)\ge 1$ for every $j\in [n]$.
\end{enumerate}
For (a), take $\bsa\in \Z_{\ge0}^n$. If $\bsa=\bfo$, then by \Cref{lem.FacetsofDegComplex}, $\Delta_\bsa(I)=\IN(G)$, which is connected as its 1-skeleton is nothing but $G^c$.
If $\bsa\in \Z_{\ge0}^n\setminus \{\bfo\}$, then $G_\bsa=\emptyset\subsetneq\Supp(\bsa)$. So it follows from \Cref{lem.FacetsofDegComplex} that either $\Delta_\bsa(I)$ is a cone or $\Delta_\bsa(I)=\emptyset$, depending on whether $\Supp(\bsa)\in\IN(G)$ or not. In any case, $\Delta_\bsa(I)$ is connected. Thus (a) is true.

For (b), assume the contrary that $\depth(S_j/I_j)=0$ for some $j\in [n]$. Then being a squarefree monomial ideal, necessarily $I_j=\langle x_i \mid i \in [n]\setminus j\rangle$, so $j$ is an isolated vertex of $G^c$. This contradicts the connectedness of $G^c$. Hence (b) is also true, and the proof is complete.
\end{proof}


\section{Invariant chains of ideals}
\label{sec.InvChain}

In this section we fix notation and provide auxiliary results on invariant chains of edge ideals. Let us begin by recalling the notion of invariant chains of ideals.

\subsection{Invariant chains of ideals}
\label{subsec.InvChain}

Let $\N$ denote the set of positive integers. As before, for each $n\in\N$, let $R_n=\kk[x_1,\dots,x_n]$ be the polynomial ring in $n$ variables over a field $\kk$. Via the natural embedding we may regard $R_n$ as a subring of $R_m$ when $m\ge n$, and thus obtain a chain of increasing polynomial rings $R_1\subset R_2\subset\cdots.$ Let $R\defas\bigcup_{n\ge1}R_n$ denote limit of this chain. Then $R=\kk[x_i\mid i\in\N]$ is a polynomial ring in infinitely many variables. Of interest are ideals in $R$ that are invariant under the action of the \emph{monoid of strictly increasing maps on $\N$}:
	\[
	\Inc = \{ \pi \colon \N \to \N \mid  \pi(n)<\pi(n+1) \text{ for all } n\geq 1\}.
	\]
This monoid acts on $R$ by means of ring endomorphisms via
	\[
	\pi \cdot x_{i}=x_{\pi(i)}
	\quad\text{for any } \pi\in \Inc \text{ and } i\ge 1.
	\]
An ideal $I\subseteq R$ is called \emph{$\Inc$-invariant} if $\pi(f)\in I$ for any $\pi\in \Inc$ and $f\in I$. Although the ring $R$ is not Noetherian, a classical result of Cohen \cite{Co67} (later rediscovered by Aschenbrenner and Hillar \cite{AH07}; see also Hillar and Sullivant \cite{HS12}) says that this ring is \emph{$\Inc$-Noetherian}, meaning that any $\Inc$-invariant ideal $I\subseteq R$ is generated by finitely many $\Inc$-orbits of polynomials.

Cohen's result has an interesting implication for chains of ideals $\Icc=(I_n)_{n\ge 1}$, where $I_n$ is an ideal in $R_n$ for $n\ge 1$. Such a chain is called \emph{$\Inc$-invariant} if
	\begin{equation}
		\label{eq-invariant}
		\langle\Inc_{m,n}(I_m)\rangle_{R_n}\subseteq I_n
		\quad\text{for all } n\ge m\ge 1,
	\end{equation}
	where $\Inc_{m,n}$ denotes the following subset of $\Inc$:
 \[
	\Inc_{m,n} = \{\pi \in \Inc \mid \pi(m) \le n\}
\]
 and $\langle\Inc_{m,n}(I_m)\rangle_{R_n}$ is the ideal in $R_n$ generated by $\Inc_{m,n}(I_m)$. When the chain $\Icc$ is $\Inc$-invariant, we say that it \emph{stabilizes} if there exists $r\ge1$ such that the inclusion in \eqref{eq-invariant} becomes an equality for all $n\ge m \ge r$. The smallest such number $r$ is called the \emph{stability index} of $\Icc$ and is denoted by $\ind(\Icc)$. A consequence of Cohen's result is that every $\Inc$-invariant chain of ideals $\Icc=(I_n)_{n\ge 1}$ always stabilizes (see \cite{HS12,KLR22}).

 The stabilization of the chain $\Icc=(I_n)_{n\ge 1}$ implies that $I_{n+1}$ can be interpreted in terms of $I_n$ for all $n\ge \ind(\Icc)$. Let us make this interpretation more explicit.
 For each integer $k\ge 0$, let $\sigma_k:\N \to \N$ be the strictly increasing map given by 
\begin{equation}
  \label{eq.sigma}
  \sigma_k(i)
=\begin{cases}
i, &\text{if $1\le i\le k$},\\
i+1, &\text{if $i\ge k+1$}.
\end{cases}
\end{equation}
It is evident that $\sigma_k\in\Inc_{n,n+1}$ for all $k\ge0$ and $n\ge1$. Denote $\sigma_k(I_n)=\langle\sigma_k(f)\mid f\in I_n\rangle_{R_{n+1}}$. Then one can easily check that $I_{n+1}=\sum\limits_{k=0}^n\sigma_k(I_n)$ for all  $n\ge\ind(\Icc).$ The next result provides a more concise representation of $I_{n+1}$. 

\begin{prop}
    \label{lem_decomposition}
Let $\Icc=(I_n)_{n\ge 1}$ be an $\Inc$-invariant chain of ideals with $\ind(\Icc)=r$. Then for any subset $\Lambda\subseteq\{0,1,\dots,n\}$ with 
$|\Lambda|=r+1$, it holds that
\[
 I_{n+1}=\sum_{k\in \Lambda}\sigma_k(I_n)
  \quad
  \text{for all } n\ge r.
\]
\end{prop}

\begin{proof}
 We proceed by induction on $n$. The case $n=r$ is clearly true since
 $I_{r+1}=\sum_{k=0}^r\sigma_k(I_r)$. Suppose that the assertion has been shown for some $n\ge r$. Let $\Lambda$ be an arbitrary subset of $\{0,1,\dots,n+1\}$ with $|\Lambda|=r+1$.
Since $I_{n+2}=\sum_{k=0}^{n+1}\sigma_k(I_{n+1})$, it suffices to check that
\[
 \sigma_l(I_{n+1})\subseteq
 \sum_{k\in \Lambda}\sigma_k(I_{n+1})
 \quad
  \text{for any}\ l\in\{0,1,\dots,n+1\}\setminus \Lambda.
\]
We fix an $l\in\{0,1,\dots,n+1\}\setminus \Lambda.$ Set 
\[
 \Lambda_1=\{k\in \Lambda\mid k>l\},
 \quad
 \Lambda_2 = \Lambda\setminus \Lambda_1,
 \quad\text{and}\quad
 \Lambda_1'=\{k-1\mid k\in \Lambda_1\}.
\]
It is evident that $k\ge l$ for all $k\in \Lambda_1'$ and $\Lambda'\defas\Lambda_1'\cup \Lambda_2$ is a subset of $\{0,1,\dots,n\}$ with $|\Lambda'|=r+1$. Thus 
$I_{n+1}=\sum_{k\in \Lambda'}\sigma_k(I_n)$
by induction hypothesis.
From \cite[Corollary 4.2]{NR17} we know that
$\sigma_j\circ\sigma_i=\sigma_i\circ\sigma_{j-1}$
whenever $j>i\ge0$. This yields
\[
 \sigma_l\circ\sigma_k=
 \begin{cases}
  \sigma_k\circ\sigma_{l-1}&\text{if } k\in \Lambda_2,\\
  \sigma_{k+1}\circ\sigma_l&\text{if } k\in \Lambda_1'.
 \end{cases}
\]
Therefore,
\begin{align*}
 \sigma_l(I_{n+1})
 &=\sigma_l\Big(\sum_{k\in \Lambda'}\sigma_k(I_{n})\Big)
 =\sum_{k\in \Lambda_1'}\sigma_l\circ\sigma_k(I_{n})
 +\sum_{k\in \Lambda_2}\sigma_l\circ\sigma_k(I_{n})\\
 &=\sum_{k\in \Lambda_1'}\sigma_{k+1}\circ\sigma_l(I_{n})
 +\sum_{k\in \Lambda_2}\sigma_k\circ\sigma_{l-1}(I_{n})\\
 &\subseteq \sum_{k\in \Lambda_1}\sigma_k(I_{n+1})
 +\sum_{k\in \Lambda_2}\sigma_k(I_{n+1})
 =\sum_{k\in \Lambda}\sigma_k(I_{n+1}).
\end{align*}
The proof is complete.
\end{proof}

Given an $\Inc$-invariant chain $\Icc=(I_n)_{n\ge 1}$, it can happen that the ideal $I_n$ is generated by a proper subset of the set of variables of $R_n$ for $n\gg0$, and the ``superfluous variables" may cause unnecessary complications. One can remove these variables by merely shifting indices. To describe this trick, let us restrict to the case of monomial ideals for simplicity.

Let $\Icc=(I_n)_{n\ge 1}$ be an $\Inc$-invariant chain of monomial ideals with $\ind(\Icc)=r$. Denote by $\Gc(I_r)$ the minimal set of monomial generators of $I_r$. We define
\begin{align*}
\msupp(I_r)&=\min\{i\mid x_i \text{ divides } u \text{ for some } u\in \Gc(I_r)\},\\
\Msupp(I_r)&=\max\{i\mid x_i \text{ divides } u \text{ for some } u\in \Gc(I_r)\}.
\end{align*}
Then $1\le \msupp(I_r) \le \Msupp(I_r)\le r$ and no element of $\Gc(I_r)$ involves variables with indices in 
$\{1,\dots,r\}\setminus \{\msupp(I_r),\dots,\Msupp(I_r)\}$. 
As the next lemma indicates, we may always reduce to the case where $\msupp(I_r)=1$ and $\Msupp(I_r)=r$ by simple index shifts. The proof of the lemma is straightforward and is therefore left to the interested reader.

\begin{lem}
\label{lem.normalizingindexshift}
Let $\Icc=(I_n)_{n\ge 1}$ be an $\Inc$-invariant chain of monomial ideals with $\ind(\Icc)=r$. Denote $i_1=\msupp(I_r)$, $p=\Msupp(I_r)$, and $\wti{r}=p-i_1+1$. Consider the chain $\wti{\Icc}=(\wti{I}_n)_{n\ge 1}$ obtained from $\Icc$ by shifting the variables by $i_1-1$ and shifting the index of $I_n$ by $r-\wti{r}$, i.e. $\wti{I}_n=0$ for $n<\wti{r}$ and 
\[
\wti{I}_n = \langle \rho(I_{n-\wti{r}+r}) \rangle_{R_{n-\wti{r}+r}} \cap R_n
\ \text{ for $n\ge \wti{r}$},
\]
where $\rho: R\to R$ is the $\kk$-endomorphism of $R$ induced by
$\rho(x_n)=0$ for $n< i_1$ and $\rho(x_n)=x_{n-i_1+1}$ for $n\ge i_1$.
Then the following hold.
\begin{enumerate}
\item
$\wti{\Icc}$ is an $\Inc$-invariant chain with $\ind(\wti{\Icc})=\wti{r}$, $\msupp(\wti{I}_{\wti{r}})=1$ and $\Msupp(\wti{I}_{\wti{r}})=\wti{r}$.
\item
$
\depth(R_n/I_n)=r-\wti{r}+\depth(R_{n+\wti{r}-r}/\wti{I}_{n+\wti{r}-r})
$
for all $n\ge r$.
\end{enumerate}
\end{lem}

\begin{ex}
\label{ex_5cycle}
Consider the chain $\Icc=(I_n)_{n\ge 1}$ with $I_n=0$ for $n<10$, 
$$
I_{10}=\langle x_2x_5,x_2x_7,x_3x_5,x_3x_9, x_7x_9\rangle
$$
is the edge ideal of a 5-cycle (see \Cref{fig.C5}), and $I_n=\langle \Inc_{10,n}(I_{10})\rangle$ for $n\ge 11$.

Then $r=\ind(\Icc)=10$, $\msupp(I_{10})=2$, $\Msupp(I_{10})=9$, and $\wti{r}=8$. So the chain $\wti{\Icc}=(\wti{I}_n)_{n\ge 1}$ is given by $\wti{I}_n=0$ for $n<8$, $\wti{I}_8=\langle x_1x_4,x_1x_6, x_2x_4, x_2x_8, x_6x_8 \rangle$, and
$\wti{I}_n=\langle \Inc_{8,n}(\wti{I}_8)\rangle$ for $n\ge 9$. Evidently, $\ind(\wti{\Icc})=8=\Msupp(\wti{I}_8)$ and $\msupp(\wti{I}_8)=1$. Moreover, $x_1,x_n$ is a regular sequence on $R_n/I_n$ and one has
$\depth(R_n/I_n)=\depth(R_{n-2}/\wti{I}_{n-2})+2$ for all $n\ge 10$.

\begin{figure}[ht]
\includegraphics[width=26ex]{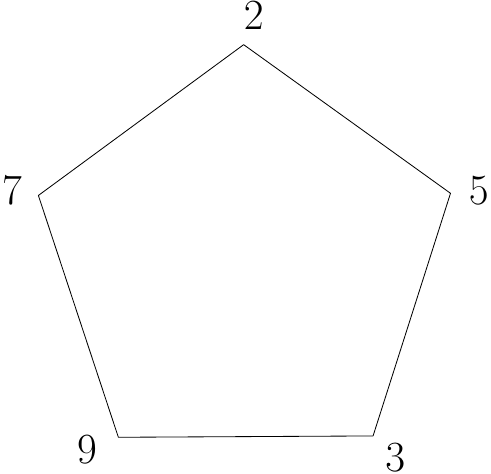}
\caption{An indexed $5$-cycle}
\label{fig.C5}
\end{figure}
\end{ex}

\subsection{Invariant chains of edge ideals}
\label{subsec.InvEd}

From now on, we focus on chains $\Icc=(I_n)_{n\ge 1}$, where each $I_n$ is an edge ideal. For convenience, the following notation will be fixed throughout the remaining part of the paper.

\begin{notn}
\label{notn_chainofedgeids}
\hfill

\begin{enumerate}
    \item 
    Let $\Icc=(I_n)_{n\ge 1}$ be an $\Inc$-invariant chain of eventually nonzero edge ideals with stability index $r=\ind(\Icc)$. For $n\ge 1$ let $G_n$ be the graph corresponding to $I_n$. 
    We always assume that $E(G_r)=\{\{i_1,j_1\},\ldots,\{i_s,j_s\}\}$ with $i_t<j_t$, $ i_1\le \cdots\le i_s$, and moreover if $i_t=i_{t+1}$ then $j_t<j_{t+1}$. Thus, in particular, 
    $\msupp(I_r)=i_1$ and $\Msupp(I_r)=\max\{j_1,\ldots,j_s\}.$ Set
\begin{align*}
  j_q&=\max\{j_t\mid i_t=i_1,\, 1\le t\le s \},\\
  p&=\Msupp(I_r)=\max\{j_1,\ldots,j_s\},\\
  b&=\min\{i_t\mid j_t=\Msupp(I_r)\}=\min\{i_t\mid j_t=p\},\\
  B&=\max\{i_t\mid j_t=\Msupp(I_r)\}=\max\{i_t\mid j_t=p\},\\
  \wti{r}&=\Msupp(I_r)-\msupp(I_r)+1=p-i_1+1.
\end{align*}

Moreover, we write $(i,j)\in E(G_n)$ if $\{i,j\}\in E(G_n)$ and $i<j$.

\item 
For $(i,j)\in \N^2$ and an integer $m\ge 0$, denote by $\Delta((i,j),m)$ the isosceles right triangle with the vertices $(i,j), (i,j+m), (i+m,j+m)$, whose legs are of length $m$.  
\end{enumerate}
\end{notn}
\begin{ex}
For the chain $\Icc$ in \Cref{ex_5cycle}, it is already known that $r=10$ and $\wti{r}=8$. As $E(G_{10})=\{(2,5),(2,7),(3,5),(3,9),(7,9)\}$, we see that
\[
i_1 =2,\ j_q=7,\ p=9, \
b=3,\ B=7.
\]
\end{ex}

We provide here some useful asymptotic properties of the graphs $G_n$. Let us first recall a simple observation from \cite[Lemma 3.3]{HNT2024} that is crucial for testing membership in $I_n$. 

\begin{lem}
\label{lem.RightTriangle}
Let $1 \le i < j \le r$ and $n \ge r$ be positive integers. Then for integers $k < l$, the following are equivalent:
 \begin{enumerate}
\item 
$x_kx_l \in \Inc_{r,n}(x_ix_j)$;
\item 
It holds that $0\le k-i\le l-j\le n-r$;
\item 
$(k,l) \in \Delta((i,j),n-r)$.
\end{enumerate}
\end{lem}

The following result shows a \emph{density} property of $G_n$ for $n\gg0$: If $(k,l)\in E(G_n)$ is identified with the point $(k,l)\in\R^2$, then moving this point in all four cardinal directions by small integral steps still yields edges of $G_n$ (see \Cref{fig:moves}).

\begin{lem}
\label{lem.shortmoves}
Let $n\ge r$ be an integer. Using \Cref{notn_chainofedgeids}, the following hold.
\begin{enumerate}
\item \textup{(Short east and short south moves)} Assume that $n\ge 3r$. Let  $k\le k'\le r$ and $n-r\le l'\le l$ be integers. If $(k,l)$ is an edge of $G_n$, then so are $(k,l')$ and $(k',l)$.
\item \textup{(Short west moves)} Let $r\le k''\le k < l$ be integers. If $(k,l)$ is an edge of $G_n$, then so is $(k'',l)$.
\item \textup{(Short north moves)} Let $k < l \le l'' \le n-r$ be integers. If $(k,l)$ is an edge of $G_n$, then so is $(k,l'')$.
\end{enumerate}
\end{lem}

\begin{figure}[ht]
\centering
\includegraphics[width=48ex]{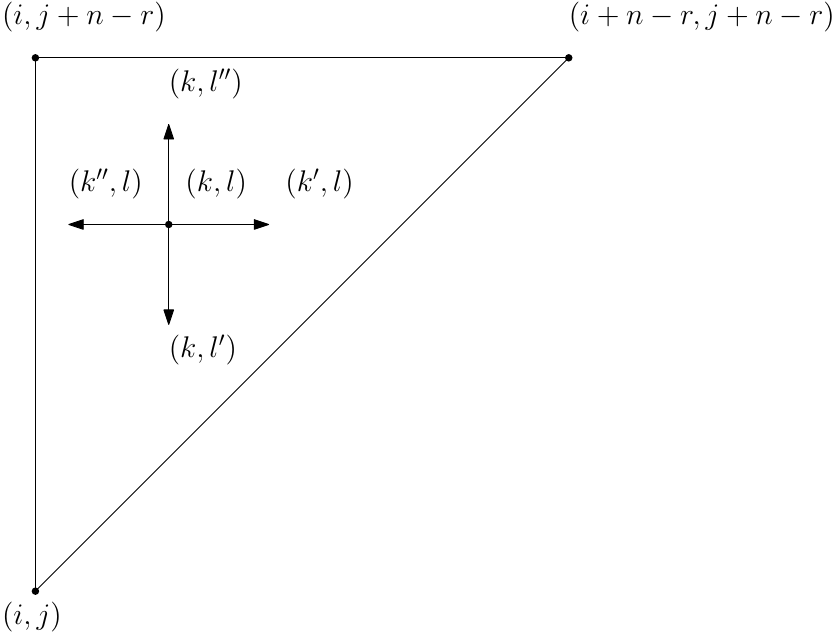}
\caption{Short moves in the four cardinal directions}
\label{fig:moves}
\end{figure}

\begin{proof}
(i) As $n\ge 3r$, we deduce that
\[
k \le k' \le r <  n-r \le l'\le l.
\]
Since $(k,l)\in E(G_n)$, it follows from \Cref{lem.RightTriangle} that $(k,l)\in \Delta((i,j),n-r)$ for some $(i,j)\in E(G_r)$ with $1\le i<j\le r$. In other words,
\begin{equation}
\label{eq.Shortmoves}
   0\le k-i\le l-j\le n-r. 
\end{equation}
We only prove that $(k,l')\in E(G_n)$; similar arguments work for $(k',l)$ as well. By \Cref{lem.RightTriangle}, it suffices to show that $(k,l') \in \Delta((i,j),n-r)$, or equivalently,
\[
0\le k-i \le l'-j\le n-r.
\]
The first and third inequalities follow immediately from \eqref{eq.Shortmoves} and the fact that $l'\le l$.
Also, the second inequality holds since $n\ge 3r$ and
\[
k-i+j \le r-1+r \le n-r \le l'.
\]

(ii) Similarly to (i), if $(k,l)\in \Delta((i,j),n-r)$ for some $(i,j)\in E(G_r)$, then we also have $(k'',l)\in \Delta((i,j),n-r)$, since $0\le r-i\le k''-i$.

The  proof of (iii) is similar and is left to the attentive reader.
\end{proof}

Our next goal is to show that $G_n$ does not contain long induced cycles for $n\gg0$. For this, we need the following consequence of \cite[Lemma 3.6]{HNT2024}.

\begin{lem}
    \label{lem.K2}
    Use \Cref{notn_chainofedgeids}. Let $n\ge 3r$ and $(u_1,v_1),(u_2,v_2)\in E(G_n)$. If $(u_1,v_1),(u_2,v_2)$ form an induced matching of $G_n$, then $[u_1,v_1]\cap[u_2,v_2]=\emptyset.$
\end{lem}

\begin{proof}
    By \Cref{lem.RightTriangle}, there exists $(i_k,j_k)\in E(G_r)$ such that $(u_k,v_k)\in\Delta((i_k,j_k),n-r)$ for $k=1,2.$ Since $n\ge 3r$, we have $n-r\ge 2r\ge 2\max\{j_1,j_2\}.$ Moreover, none of the pairs $\{u_1,u_2\}$, $\{u_1,v_2\}$, $\{u_2,v_1\}$, $\{v_1,v_2\}$ is an edge of $G_n$ since $(u_1,v_1),(u_2,v_2)$ form an induced matching of $G_n$. So if we assume without loss of generality that $u_1<u_2$, then it follows from \cite[Lemma 3.6]{HNT2024} that
    \[
    v_1<i_2<n-r+j_1<u_2.
    \]
    Hence, $[u_1,v_1]\cap[u_2,v_2]=\emptyset$, as desired.
\end{proof}

We are now ready to prove the following.

\begin{lem}
\label{lem.noinducedCm>5}
Use \Cref{notn_chainofedgeids}. Then for $n\ge 3r$, $G_n$ has no induced cycle $C_m$ with $m\geq 6$.
\end{lem}

\begin{proof}
Assume on the contrary that for some $n\ge 3r$, $G_n$ contains an induced cycle $C_m$ with $m\geq 6$. Label the vertices of $C_m$ as $u_1,\ldots,u_m$ such that $u_1=\min\{u_1,\ldots,u_m\}.$
For two real numbers $x,y$, let $(x,y)^{\le}$ be the ordered pair $(\min\{x,y\},\max\{x,y\})\in \R^2$, and $[x,y]^{\le}$ be the closed interval $[\min\{x,y\},\max\{x,y\}]\subseteq \R$.
Since $m\geq 6$, $\{(u_1,u_2),(u_{m-2},u_{m-1})^{\le}\}$ is an induced matching of $G_n$. Thus, $[u_1,u_2] \cap [u_{m-2},u_{m-1}]^{\le}=\emptyset$ by \Cref{lem.K2}. It follows that
$$
u_1<u_2< \min\{u_{m-2},u_{m-1}\}.
$$ 
Analogously, $u_1<u_m<\min\{u_3,u_4\}$ as $\{(u_1,u_m),(u_3,u_4)^{\le}\}$ is an induced matching of $G_n$. So if $u_2<u_m$, then $u_2<u_m<u_3$. Otherwise, if $u_m<u_2$, then $u_m<u_2<u_{m-1}$. In either case, it always holds that $[u_2,u_3]^{\le} \cap [u_{m-1},u_m]^{\le} \neq \emptyset$. Thus, $\{(u_2,u_3)^{\le},(u_{m-1},u_{m})^{\le}\}$ is not an induced matching of $G_n$ by \Cref{lem.K2}. This contradiction concludes the proof.
\end{proof}

For the graph $G_n\setminus N[n]$, which arises from the ideal $I_n:x_n$, a stronger result holds true.

\begin{lem}
\label{lem.GWeaklyChordal}
Using \Cref{notn_chainofedgeids}, assume that $p=r$. Then the following statements hold for all $n\ge 2r+1$.
\begin{enumerate}
    \item 
    $V(G_n\setminus N[n])=\{1\ldots,b-1\} \cup  \{n-r+B+1,\ldots,n-1\}$.
    \item 
    The graphs $G_n\setminus N[n]$ and $G_{n+1}\setminus N[n+1]$ are isomorphic.
    \item 
    $G_n\setminus N[n]$ is weakly chordal.
\end{enumerate}
\end{lem}

\begin{proof}
(i) Since $2r+1\ge r+B-b$, it suffices to show that 
\[
N(n)=\{b,b+1,\ldots,n-r+B\}
\quad \text{for } n\ge r+B-b.
\]
Let $k\in N(n)$. Then $(k,n)\in E(G_n)$. By \Cref{lem.RightTriangle}, there exists $(i,j)\in E(G_r)$ such that $(k,n)\in \Delta((i,j),n-r)$, i.e.
\[
0\le k-i\le n-j\le n-r.
\]
This implies $i\le k\le n-j+i$ and $j\ge r$. Since $j\le r$, we get $j=r$ and thus $(i,r)\in E(G_r)$. Note that $p=r$. So the definition of $b$ and $B$ in \Cref{notn_chainofedgeids} gives $b\le i\le B$. Hence, $b\le i\le k \le n-r+i\le n-r+B$ and therefore $k\in \{b,b+1,\ldots,n-r+B\}$.

Conversely, if $b\le k\le n-r+B$, we see that $(k,n)\in \Delta((b,r),n-r)$ if $b\le k\le n-r+b$ and  $(k,n)\in \Delta((B,r),n-r)$ if $B\le k\le n-r+B$.  As $n\ge r+B-b$, we deduce $k\in N(n)$, as desired.

(ii) Denote $F_n=G_n\setminus N[n]$. For $n\ge 2r+1$, it follows from (i) that
\begin{equation}
\label{I_n:x_n}
I_n:x_n
=\langle x_b,x_{b+1},\ldots,x_{n-r+B}\rangle +L_n,
\end{equation}
where $L_n\subseteq R_n$ is an edge ideal supported on $V(F_n)$. Moreover, $L_n$ is exactly the edge ideal of $F_n$ if it is viewed as an ideal in the ring $\kk[x_i\mid i\in V(F_n)]$. Recall the map $\sigma_k$ defined in \eqref{eq.sigma}. It is apparent that $\sigma_b$ induces a bijective map from $V(F_n)$ to $V(F_{n+1})$. We show that this map is a graph isomorphism between $F_n$ and $F_{n+1}$ for all $n\ge 2r+1$. In other words, we need to prove that $L_{n+1}=\sigma_b(L_n)$, or equivalently,
\[
I_{n+1}:x_{n+1}=\langle x_b,x_{b+1},\ldots,x_{n-r+B+1}\rangle +\sigma_b(L_n)
\quad
  \text{for all } n\ge 2r+1.
\]
Indeed, it follows from \cite[Lemma 4.7]{LN22} that the chain $(I_n:x_n)_{n\ge 1}$ is $\Inc$-invariant with stability index at most $r+1$. So by \Cref{lem_decomposition},
\begin{equation}
\label{sigmaI_n:x_n}
I_{n+1}:x_{n+1}=\sum_{k\in \Lambda}\sigma_k(I_n:x_n)
  \quad
  \text{for all } n\ge r+1,
\end{equation}
where $\Lambda$ is any subset of $\{0,1,\dots,n+1\}$ with $|\Lambda|=r+2$. Choose 
\[
\Lambda=\{b,b+1,\ldots,b+r+1\}.
\]
Then $\Lambda\subseteq\{b,b+1,\ldots,n-r+B\}$ for all $n\ge 2r+1$. Take any $x_ux_v\in L_n$ with $u<v$. Since $u,v\in V(F_n)$, only the following cases can occur.

\emph{Case 1}: $u<v< b$. In this case, $\sigma_k(x_ux_v)=x_ux_v$ for all $k\in \Lambda$.

\emph{Case 2}: $u< b<n-r+B<v$. In this case, $\sigma_k(x_ux_v)=x_ux_{v+1}$ for all $k\in \Lambda$.

\emph{Case 3}: $n-r+B<u<v.$ In this case, $\sigma_k(x_ux_v)=x_{u+1}x_{v+1}$ for all $k\in \Lambda$.

It follows that $\sigma_k(L_n)=\sigma_b(L_n)$ for all $k\in \Lambda$. Thus from \eqref{I_n:x_n} and \eqref{sigmaI_n:x_n} we get
\begin{align*}
   I_{n+1}:x_{n+1}
   &=\sum_{k=b}^{b+r+1}\sigma_k\big(\langle x_b,x_{b+1},\ldots,x_{n-r+B}\rangle +L_n\big)\\
   &=\langle x_b,x_{b+1},\ldots,x_{n-r+B+1}\rangle +\sum_{i=b}^{b+r+1}\sigma_i(L_n)
   \\
   &=\langle x_b,x_{b+1},\ldots,x_{n-r+B+1}\rangle +\sigma_b(L_n)
\end{align*}
for all $n\ge 2r+1$, as wanted. 
The desired assertion follows.

(iii) In view of (ii), it is enough to show that $F_n$ is weakly chordal for some $n\gg0$. Let $n\ge \max \{5r, r(r-B+b-2) \}$. Since $F_n$ is an induced subgraph of $G_n$, it follows from \Cref{lem.noinducedCm>5} that $F_n$ contains no induced cycle $C_m$ for $m\ge 6$. By \cite[Proposition 4.1]{HNT2024}, the complement $F_n^c$ also contains no induced cycle $C_m$ for $5\le m \le {n}/{r}$. As
\[
\frac{n}{r}\ge r-B+b-2= |V(F_n)| = |V(F_n^c)|,
\]
$F_n^c$ has no induced cycle $C_m$ for all $m\ge 5$. 
This implies that $F_n$ also has no induced $C_5$ since $C_5^c=C_5$. Therefore, $F_n$ is weakly chordal. The proof is complete.
\end{proof}

\begin{rem}
    By index shift, \Cref{lem.GWeaklyChordal} is valid even without the assumption that $p=r$. In this case, the graph $G_n\setminus N[n]$ would have to be replaced by $G_n\setminus N[n-r+p]$ and some indices in the result would have to be modified accordingly. The details are left to the reader.
\end{rem}

We conclude this section with a strengthened version of \cite[Proposition 7.4]{HNT2024}, in which the index of regularity stability is slightly reduced. 

\begin{lem}
\label{lem.reg1}
Use \Cref{notn_chainofedgeids}. If $j_q=p$, i.e. $x_{i_1}x_p\in I_r$, then
\[
\reg I_n =2
\quad\text{for all } n\ge 3r-3.
\]
\end{lem}

\begin{proof}
The result actually follows implicitly from the proof of \cite[Proposition 7.4]{HNT2024}. We just need to tighten an inequality in that proof. Indeed, 
by  Fr\"oberg's theorem \cite[Theorem 1]{Fr90}, we have to show that $G_{n+r}$ is \emph{cochordal} (i.e. $G_{n+r}^c$ is chordal) for all $n\ge 2r-3$. If $G_{n+r}$ is  not cochordal for some $n\ge 2r-3$, then it is shown in the proof of \cite[Proposition 7.4]{HNT2024} that there exists some $k\in [s]$ such that
\[
n<r+j_q-i_q-j_k,
\]
which yields $n<2r-3$ since $j_q\le r$, $i_q\ge1$ and $j_k>i_k\ge1$. This contradiction concludes the proof.
\end{proof}


\section{Sparsity index and bounds for the asymptotic depth}
\label{sec.Bound}

This section can be seen as the starting point for the proof of  \Cref{thm_limdepth_main}, which will be completed in \Cref{sec.MaxDepth,sec.MinDepth}. We provide here upper and lower bounds for $\depth(R_n/I_n)$ when $n\gg 0$. As we will see in the subsequent sections, $\depth(R_n/I_n)$ always attains one of these bounds for $n\gg 0$. Throughout the section, we continue to use \Cref{notn_chainofedgeids}.

Let us first introduce the following crucial notion.

\begin{defn}
We call 
\[
\spi(I_r)=\min \{j-i \mid (i,j)\in E(G_r)\}
=\min \{j_t-i_t \mid 1\le t \le s\}
\]
the \emph{sparsity index} of $I_r$. 
\end{defn}

The $\Inc$-invariance of the chain $\Icc=(I_n)_{n\ge 1}$ implies that $\spi(I_n)=\spi(I_r)$ for all $n\ge r$. We can therefore set $$\spi(\Icc)\defas\spi(I_r)$$
and call this number the \emph{sparsity index} of $\Icc$. The nomenclature is justified by the observation that the bigger $\spi(\Icc)$ is, the fewer edges each graph $G_n$ may have: 

\begin{rem}
\label{rem.sparsityindandedges}
Let $n\ge r$ and $1\le i < j\le n$ be integers. If $(i,j)\in E(G_n)$, then it is evident that $j-i\ge \spi(I_n)=\spi(\Icc)$. Thus, $i$ and $j$ are not adjacent in $G_n$ whenever $j-i< \spi(\Icc)$.
\end{rem}

\begin{ex}
    The chain $\Icc$ in \Cref{ex_5cycle} has $\spi(\Icc)=2$. 
\end{ex}

The next result gives an upper bound for $\depth(R_n/I_n)$ when $n\gg 0$. Recall that 
$$\wti{r}=\Msupp(I_r)-\msupp(I_r)+1=p-i_1+1$$
is the stability index of the chain $\wti{\Icc}$ considered in \Cref{lem.normalizingindexshift}.

\begin{thm}
\label{thm_upperbound_limdepth}
For all $n\ge r+2\wti{r}$, it holds that
\[
\depth(R_n/I_n) \le r-\wti{r}+\spi(\Icc).
\]
\end{thm}

\begin{proof}
Define the chain $\wti{\Icc}$ as in \Cref{lem.normalizingindexshift}. Observe that $\spi(\Icc)=\spi(\wti{\Icc})$. So passing to the chain $\wti{\Icc}$ we may assume $i_1=1$, $p=r$, and in this case, what we need to show becomes
\[
\depth(R_n/I_n) \le \spi(\Icc)
\quad\text{for all }
n\ge 3r.
\]
Put $m=\spi(\Icc)$. By \Cref{prop_SmallDepthCrit}, it suffices to provide a subset $F\subseteq[n]$ with $|F|=m$ such that $\depth((R_n)_{F}/(I_n)_{F})=0$. Since $(I_n)_{F}$ is squarefree, the last condition is equivalent to $(I_n)_{F}=\mm_F$, where $\mm_F$ is the graded maximal ideal of $(R_n)_{F}$.
From the assumption that $i_1=1, p=r$ and $m=\spi(\Icc)$, we deduce that $E(G_r)$ contains (not necessarily distinct) edges of the forms $(1,a), (v-m,v), (b,r)$, where $m+1 \le a, v\le r$ and $1\le b\le r-m$. Let 
\[
F=\{\alpha,\alpha-1,\ldots,\alpha-m+1\}
\ \text{ with }\
\alpha=\max\{a+v,b\}\le 2r.
\]
We show that $F$ has the desired property for all $n\ge 3r$. Obviously, $|F|=m$. So it remains to prove that $(I_n)_{F}=\mm_F$. This is done through the following claims.

\begin{claim}
$(I_n)_{F}\subsetneq(R_n)_{F}$. 
\end{claim}

Indeed, if $(I_n)_{F}=(R_n)_{F}$, then $G_n$ must have an edge of the form $(k,l)$, where $k< l$ are both in $F$. But then $l-k\le m-1$, contradicting \Cref{rem.sparsityindandedges}.

\begin{claim}
    $(I_n)_{F}\supseteq \mm_F$, i.e. each vertex $k\in [n]\setminus F$ is adjacent to a vertex $l\in F$ in $G_n$.
\end{claim}

In fact, we have $[n]\setminus F=\{1,\dots,\alpha-m\} \cup  \{\alpha+1,\dots,n\}$. Using \Cref{lem.RightTriangle}, it suffices to show that for each $k\in [n]\setminus F$, there exists $l\in F$ such that either $(k,l)$ or $(l,k)$ belongs to one of the triangles $\Delta((1,a),n-r)$, $\Delta((v-m,v),n-r)$ and $\Delta((b,r),n-r).$ We distinguish the following cases.

\emph{Case 1}: $1\le k\le v-m$. Then it is clear that 
\[
0\le k-1\le \alpha-a\le n-r
\]
since $v\le \alpha-a$ and $\alpha\le 2r\le n-r$. Hence, $(k,\alpha)\in\Delta((1,a),n-r).$

\emph{Case 2}: $\alpha+r-b\le k\le n$. In this case, $(\alpha,k)\in\Delta((b,r),n-r)$ since 
\[
0\le \alpha-b\le k-r\le n-r.
\]

\emph{Case 3}: $v-m< k\le \alpha-m$ or $\alpha+1\le k< \alpha+r-b$. Set
\[
l=
\begin{cases}
    k-m & \text{if } \alpha+1\le k\le \alpha+m,\\
    \alpha & \text{otherwise}.
\end{cases}
\]
Then $l\in F$. Similarly to the previous cases, one can show that either $(k,l)$ (when $k<l$) or $(l,k)$ (when $k>l$) belongs to $\Delta((v-m,v),n-r)$. The details are left to the reader.
\end{proof}

Let us now provide a lower bound for $\depth(R_n/I_n)$ when $n\gg0$.

\begin{thm}
\label{thm_lowerbound_limdepth}
For all $n\ge r$, the following inequality holds
\[
\depth(R_n/I_n)\ge r-\wti{r}+\min\{\spi(\Icc), 2\}.
\]
\end{thm}

\begin{proof}
Using \Cref{lem.normalizingindexshift}, we may assume that $i_1=1$ and $p=r$. In this case, we need to show that 
\[
\depth(R_n/I_n)\ge \min\{\spi(\Icc), 2\}
\quad\text{for all }
n\ge r.
\]
Since $I_n$ is a squarefree non-maximal ideal of $R_n$, it always holds that $\depth(R_n/I_n)\ge 1$. Hence, it suffices to prove that $\depth(R_n/I_n)\ge 2$ when $\spi(\Icc)\ge 2$, which we will assume from now. By \Cref{cor_depthge2crit}, we need to show that $G_n^c$ is connected for each $n\ge r$. Indeed, take two arbitrary vertices $k,l$ of $G_n^c$ with $k<l$. These vertices are joined by the edges $(k,k+1), (k+1,k+2),\ldots, (l-1,l)$, which all belong to $G_n^c$ since $\spi(\Icc) \ge 2$. Hence, $G_n^c$ is connected, as desired.
\end{proof}

As a direct consequence of \Cref{thm_upperbound_limdepth,,thm_lowerbound_limdepth} we obtain the following.

\begin{cor}
\label{cor.limdepth.sple2}
Assume that $\spi(\Icc)\le 2$. Then for all $n\ge r+2\wti{r}$, it holds that
\[
\depth(R_n/I_n)= r-\wti{r}+\spi(\Icc).
\]
\end{cor}


\section{Maximal asymptotic depth}
\label{sec.MaxDepth}

In this section we show that the upper bound given in \Cref{thm_upperbound_limdepth} is attained when $j_q=p$, where we use \Cref{notn_chainofedgeids} throughout as usual. The following result covers \Cref{thm_limdepth_main}(i).

\begin{thm}
\label{thm.limdepth.maximal}
Assume that $j_q=p$. Then for all $n\ge r+2\wti{r}$, one has
 \[
 \depth (R_n/I_n) =r-\wti{r}+\spi(\Icc).
 \]
\end{thm}

To prove this theorem, we need some auxiliary results. We first give a lower bound for the size of a maximal independent set of the graph $G_n$ when $n\gg0$.

\begin{lem}
\label{lem.reg2normal.indepsets}
Assume that $j_q=p$. Then for all $n\ge r+2\wti{r}$, every maximal independent set of $G_n$ has size at least $ \spi(\Icc)$. 
\end{lem}

\begin{proof}
Set $m=\spi(\Icc)$. Using \Cref{lem.normalizingindexshift} we may assume that $i_1=1$ and $j_q=r$. In this case, it is enough to show that every maximal independent set of $G_n$ has size at least $m$ for all $n\ge 3r$.
Suppose to the contrary that $G_n$ has a maximal independent set $U$ of size at most $m-1$. Denote $\alpha=\min U$ and $\beta=\max U$. To derive a contradiction, let us prove the following claims.

\begin{claim}
    $\alpha\le r-1$.
\end{claim}

Indeed, consider the following sets of size $m$:
\begin{align*}
    V_1&=\{1,2,\ldots,m\},\\
    V_2&=\{\beta-m+1,\beta-m+2,\ldots,\beta\},\\
    V_3&=\{\alpha,\alpha+1,\ldots,\alpha+m-1\}.
\end{align*}
By \Cref{rem.sparsityindandedges}, $V_1$ is an independent set of $G_n$. Since $|U|<|V_1|$, it follows from the maximality of $U$ that $U\nsubseteq V_1$. This yields $\beta\ge m+1$ and thus $V_2\subseteq [n]$. Again by \Cref{rem.sparsityindandedges}, $V_2$ is an independent set of $G_n$.  Hence, $U\nsubseteq V_2$ due to the maximality of $U$. It follows that $\alpha\le \beta-m$, and consequently, $\alpha+m\le \beta\le n$. Thus,  $V_3\subseteq [n]$. Since $|U|<|V_3|$, there exists $i\in [m-1]$ such that $\alpha+i\notin U$. The maximality of $U$ implies that $U\cup \{\alpha+i\}$ is a dependent set of $G_n$. 
Therefore, $\{\alpha+i,u\}\in E(G_n)$ for some $u\in U$. Note that $|u-(\alpha+i)|\ge m$ by \Cref{rem.sparsityindandedges}. Since $\alpha \le u$ and $i\le m-1$, we must have $\alpha+i<u.$ 
Now if $\alpha\ge r$, then moving west from $(\alpha+i,u)$ to $(\alpha,u)$ using \Cref{lem.shortmoves}, we get $(\alpha,u)\in E(G_n)$. This contradicts the independence of $U$. Hence, $\alpha\le r-1$, as claimed.

\begin{claim}
    $\beta\ge n-r+1$.
\end{claim}

We argue similarly as above. Since $U\nsubseteq V_2$, there exists $j\in [m-1]$ such that $U\cup \{\beta-j\}$ is a dependent set of $G_n$. Thus, $\{\beta-j,v\}\in E(G_n)$ for some $v\in U$. Using  \Cref{rem.sparsityindandedges} together with the fact that $v\le \beta$ and $j\le m-1$, we also deduce that $v<\beta-j$.
If $\beta\le n-r$, then moving north from $(v,\beta-j)$ to $(v,\beta)$ using \Cref{lem.shortmoves}, we get $(v,\beta)\in E(G_n)$. This again contradicts the independence of $U$. Hence, $\beta\ge n-r+1$. 

\begin{claim}
\label{cl.Minmax}
    $(\alpha,\beta)\in E(G_n)$.
\end{claim}

The assumption that $i_1=1$ and $j_q=r$ implies $(1,r)\in E(G_r)$. So by \Cref{lem.RightTriangle}, it suffices to show that
$(\alpha,\beta)\in \Delta((1,r),n-r)$. Indeed, this follows from
\[
0\le \alpha-1 \le r-2 \le n-2r+1 \le \beta-r\le n-r,
\]
where the third inequality holds since $n\ge 3r$. 

\Cref{cl.Minmax} contradicts the independence of $U$ and thus completes the proof.
\end{proof}

The following technical lemma also plays a role in the proof of \Cref{thm.limdepth.maximal}.

\begin{lem}
\label{lem.reg2normal.connectedness}
Assume that $j_q=p$ and $\spi(\Icc)\ge 2$. Let $U$ be an independent set of $G_n$ of size at most $\spi(\Icc)-2$. Then the complement of $G_n\setminus N[U]$ is connected for all $n\ge r+2\wti{r}$.
\end{lem}

\begin{proof}
Using \Cref{lem.normalizingindexshift} we may assume that $i_1=1$ and $j_q=r$. In this case, we need to show that the graph 
$F_n\defas (G_n\setminus N[U])^c=G_n^c\setminus N[U]$
is connected for all $n\ge 3r$. Suppose to the contrary that $F_n$ is not connected for some $n\ge 3r$. Let $i<j$ be vertices of $F_n$ that are not joined by a path in $F_n$ with $j-i$ being minimal. Then in particular, $(i,j)\in E(G_n)$ since $(i,j)\not\in E(G_n^c)$. Moreover, any $k\in [n]$ with $i<k<j$ is not a vertex of $F_n$ due to the minimality of $j-i$. Thus, \begin{equation}
\label{eq.boundingj}
\{i+1,\dots,j-1\}\subseteq N[U].
\end{equation}
Denote $m=\spi(\Icc)$, $d=|U|$, $\alpha=\min U$, and $\beta=\max U$. Then it is clear that $\beta-\alpha\ge d-1$. Also, $j-i\ge m$ as $(i,j)\in E(G_n)$. 
We distinguish two cases.

\emph{Case 1}:  $i>\alpha$. A contradiction will be derived from the following claims.

\begin{claim}
    \label{cl.alphaj}
    $\alpha \le r-1$ and $j\le n-r.$
\end{claim}

If $\alpha\ge r$, then moving west from $(i,j)$ to $(\alpha, j)$ using \Cref{lem.shortmoves}, we get $(\alpha, j)\in E(G_n)$, contradicting the fact that $j\notin N[U]$. Hence, $\alpha\le r-1$. By assumption, $(1,r)\in E(G_r)$. So \Cref{lem.RightTriangle} yields $(\alpha,j)\notin \Delta((1,r),n-r)$, i.e. at least one of the inequalities
\[
0\le \alpha -1 \le j-r \le n-r
\]
is not true. We see that the middle inequality must be false. Thus
\[
j< \alpha-1+r \le 2r-2\le n-r,
\]
where the last inequality follows from $n\ge 3r$. 

\begin{claim}
    \label{cl.betaj}
    $\beta<j$. 
\end{claim}

Suppose that $j<\beta$. If $\beta\le n-r$, then moving north from $(i,j)$ to $(i,\beta)$ using \Cref{lem.shortmoves}, we deduce $(i,\beta)\in E(G_n)$. This contradicts the fact that $i\notin N[U]$. On the other hand, if $\beta\ge n-r+1$, then the hypothesis $n\ge 3r$ and the inequality $\alpha\le r-1$ from Claim 1 give
\[
0\le \alpha-1 \le r-2 < n-2r+1 \le \beta-r \le n-r.
\]
This implies $(\alpha,\beta)\in \Delta((1,r),n-r)$, contradicting the independence of $U$.

\begin{claim}
\label{cl.uj}
    $(u,j-h)\in E(G_n)$ for some $u\in U$ and $h\in [d+1]$.
\end{claim}

As $|U|=d$, there exists $h\in [d+1]$ such that $j-h\notin U$. Since $d+1\le m-1\le j-i-1$, it is easily seen from \eqref{eq.boundingj} that $j-h\in N[U]$. Thus, $\{u,j-h\}\in E(G_n)$ for some $u\in U$. This implies $|(j-h)-u|\ge m$. Since $u\le \beta<j$ and $h\le d+1\le m-1$, it must hold that $u<j-h$. Therefore, $(u,j-h)\in E(G_n)$.

Let us now derive a contradiction. Since $j\le n-r$ and $(u,j-h)\in E(G_n)$, it follows from \Cref{lem.shortmoves}(iii) that $(u,j)\in E(G_n)$. This contradicts the fact that $j\notin N[U]$, as wanted.

\emph{Case 2}: $i<\alpha$. Let us first prove the following claims.

\begin{claim}
\label{cl.iu}
    $(i+h,u)\in E(G_n)$ for some $u\in U$ and $h\in [d+1]$.
\end{claim}

We argue analogously to the proof of \Cref{cl.uj}. By \eqref{eq.boundingj} and the fact that $d+1\le j-i-1$, we may choose $h\in [d+1]$ such that $i+h\in N[U] \setminus U.$ This implies $\{i+h,u\}\in E(G_n)$ for some $u\in U$. We must have $i+h<u$ since $|(i+h)-u|\ge m$, $i<\alpha\le u$ and $h\le d+1\le m-1$. Hence, $(i+h,u)\in E(G_n)$.

\begin{claim}
\label{cl.alphav}
    If $\alpha+d+1<j$, then $(\alpha+k,v)\in E(G_n)$ for some $v\in U$ and $k\in [d+1]$.
\end{claim}

From the assumption $\alpha+d+1<j$ it follows that $\alpha+k\in N[U]$ for all $k\in [d+1]$. The rest of the argument is the same as in the proof of \Cref{cl.iu} and is omitted.

\begin{claim}
\label{cl.wj}
    If $\beta<j$, then $(w,j-l)\in E(G_n)$ for some $w\in U$ and $l\in [d+1]$.
\end{claim}
 
The argument is the same as in the proof of \Cref{cl.uj} and is omitted.

Let us now derive a contradiction. If $i\ge r$, then using \Cref{cl.iu} and \Cref{lem.shortmoves}(ii) we infer that $(i,u)\in E(G_n)$. This contradicts the fact that $i\notin N[U]$. Hence, $i\le r-1$. Notice that $(i,\beta)\notin \Delta((1,r),n-r)$ by \Cref{lem.RightTriangle}. So arguing as in the proof of \Cref{cl.alphaj} yields
$$\beta< i+r-1 \le 2r-2\le n-r.$$
If $j< \beta$, then moving north from $(i,j)$ to $(i,\beta)$ using \Cref{lem.shortmoves}, we get the contradiction that $(i,\beta)\in E(G_n)$. Thus, $\beta<j$. If $j\le n-r$, then \Cref{cl.wj} together with \Cref{lem.shortmoves}(iii) implies that $(w,j)\in E(G_n)$, again a contradiction. Hence, $j\ge n-r+1$. As $(\alpha,j)\notin E(G_n)$, also $(\alpha,j)\notin \Delta((1,r),n-r)$, and it follows that  $\alpha > j-r+1 \ge n-2r+2 >  r$. Moreover, $\alpha+d+1<j$ since
\[
\alpha+d-1\le \beta<2r-2\le (n-r+1)-3\le j-3.
\]
So using \Cref{cl.alphav} and \Cref{lem.shortmoves}(ii) we deduce that $(\alpha,v)\in E(G_n)$. This contradiction concludes the proof.
\end{proof}

We are now ready to prove \Cref{thm.limdepth.maximal}.

\begin{proof}[Proof of \Cref{thm.limdepth.maximal}]
By \Cref{cor.limdepth.sple2}, it suffices to consider the case $m\defas\spi(\Icc)\ge 3$. Moreover, using \Cref{lem.normalizingindexshift} we may assume that $i_1=1$ and $ j_q=p=r$. In this case, by \Cref{thm_upperbound_limdepth}, it is enough to show that 
\[
 \depth (R_n/I_n) \ge m
 \quad\text{for all }
 n\ge 3r.
 \]
 Using Takayama's formula, this is equivalent to proving that
 \[
 \wti{H}_{i-|G_\bsa|-1}(\Delta_\bsa(I_n))=0 
 \quad \text{for all $n\ge 3r,$ all $i\le m-1$ and all $\bsa \in \Z^n$}.
 \]
 By \Cref{lem.FacetsofDegComplex}(ii), we may assume that $\bsa\in \{-1,0,1\}^n$. Moreover, it suffices to examine the case $\Supp(\bsa)=G_\bsa\in \IN(G_n)$ by virtue of \Cref{lem.FacetsofDegComplex}(iii)--(iv). In this case, $\bsa \in \{-1,0\}^n$. 
As $j_q=p=r$, we know from \Cref{lem.reg1} that $\reg(R_n/I_n)=1$ for all $n\ge 3r$. Thus, if $i+\sum_{i=1}^n a_i=i-|G_\bsa| \ge 2$, then Takayama's formula gives
\[
\wti{H}_{i-|G_\bsa|-1}(\Delta_\bsa(I_n))\cong H^i_\mm(R_n/I_n)_\bsa=0.
\]
Therefore, we may assume that $|G_\bsa|\ge i-1$. If $|G_\bsa|\ge i+1$, then $\wti{H}_{i-|G_\bsa|-1}(\Delta_\bsa(I_n))=0$ since $i-|G_\bsa|-1\le -2$. So there are only two cases left:

\emph{Case 1}: $|G_\bsa|=i$. Since $i\le m-1$, it follows from \Cref{lem.reg2normal.indepsets} that $G_\bsa$ is not a maximal independent set of $G_n$. Hence, $\Delta_\bsa(I_n)\neq \{\emptyset\}$ by \Cref{lem.FacetsofDegComplex}(v). This implies
\[
 \wti{H}_{i-|G_\bsa|-1}(\Delta_\bsa(I_n))=\wti{H}_{-1}(\Delta_\bsa(I_n))=0.
\]

\emph{Case 2}: $|G_\bsa|=i-1$. Recall from \Cref{lem.FacetsofDegComplex} that
\[
\Fc(\Delta_\bsa(I_n)) =\{F \setminus G_\bsa \mid G_\bsa\subseteq F\subseteq [n], F\in \Fc(\IN(G_n))\}.
\]
Thus, the 1-skeleton of $\Delta_\bsa(I_n)$ is exactly the graph $(G_n\setminus N[G_\bsa])^c$. Since $|G_\bsa|=i-1\le m-2$, \Cref{lem.reg2normal.connectedness} says that $(G_n\setminus N[G_\bsa])^c$ is connected. Hence,
\[
 \wti{H}_{i-|G_\bsa|-1}(\Delta_\bsa(I_n))=\wti{H}_{0}(\Delta_\bsa(I_n))=0,
 \]
 as desired.
\end{proof}

We conclude this section with an example illustrating that the lower bound for the stability index of $\depth(R_n/I_n)$ given in \Cref{thm.limdepth.maximal} could be close to optimal.

\begin{ex}
    \label{ex.depthStabJq=r}
    Consider the chain $\Icc=(I_n)_{n\ge 1}$ with stability index $r\ge 6$ and 
\[
E(G_r)=\{(1,r), (2,4), (3,5)\}.
\]
This chain satisfies $i_1=1$ and $j_q=p=r$.
Computations with Macaulay2 \cite{GS96} suggest that
\begin{align*}
 \depth(R_n/I_n) &= \begin{cases}
                    2 &\text{if $n= 2r-5$},\\
                    3 &\text{if $2r-4\le n \le 3r-10$},\\
                    2 &\text{if $n\ge 3r-9$}.
                    \end{cases}                                 
\end{align*}
Assuming the above result, we see that the lower bound $n\ge 3r$ given in \Cref{thm.limdepth.maximal} for the stability index of $\depth(R_n/I_n)$ cannot be improved to $n\ge 3r-10$ in general.   
\end{ex}


\section{Minimal asymptotic depth}
\label{sec.MinDepth}

Our goal in this section is to complete the proof of \Cref{thm_limdepth_main} and thereby provide a comprehensive picture of the asymptotic behavior of $\depth(R_n/I_n)$: we prove the following slight generalization of \Cref{thm_limdepth_main}(ii), showing that the lower bound given in \Cref{thm_lowerbound_limdepth} is attained when $j_q<p$. As always, \Cref{notn_chainofedgeids} is used throughout the section.

\begin{thm}
\label{thm.limdepth.minimal}
 Assume that $j_q<p$. Then 
 \[
 \depth (R_n/I_n) =r-\wti{r}+\min\{2, \spi(\Icc)\}
 \]
for all $n\ge r+2\wti{r}$.
\end{thm}

 The proof of this theorem is mainly based on the following nonvanishing result for the first reduced homology group of the independence complex $\IN(G_n)$ of the graph $G_n$. A complete description of all reduced homology groups of $\IN(G_n)$ can be found in \Cref{sec.homology}.

\begin{prop}
\label{prop.H1nonvanishing}
Assume that $i_1=1$ and $p=r$. If $j_q< r$ and $\spi(\Icc)\ge 2$, then 
\[
\wti{H}_1(\IN(G_n))\cong \kk 
\quad \text{for all } n\ge 3r.
\] 
\end{prop}

The main idea to prove \Cref{prop.H1nonvanishing} is to proceed by induction on $r-j_q$ using the long exact sequence in \Cref{lem.MayerVietoris}. Let us begin by showing (non)vanishing results for the zeroth and first reduced homology groups of the independence complexes $\IN(G_n\setminus n)$ and $\IN(G_n\setminus N[n])$.

\begin{lem}
    \label{lem.H0Gminusn}
    Assume that $i_1=1$ and $p=r$. If $\spi(\Icc)\ge 2$, then
    \[
    \wti{H}_{0}(\IN(G_n\setminus n))=0
    \quad \text{for all } n\ge r.
    \]
\end{lem}

\begin{proof}
    Obviously, the graph $G_n\setminus n$ contains at least one vertex for all $n\ge r$. So by \Cref{lem.vanishingofH0}, it suffices to verify the connectedness of the complementary graph $(G_n\setminus n)^c$. But this is clear, because any two vertices $i,j$ of $(G_n\setminus n)^c$ with $i<j$ are connected by the edges $(i,i+1), (i+1,i+2),\ldots, (j-1,j)$, all of which belong to $(G_n\setminus n)^c$ since $\spi(\Icc)\ge 2$.
\end{proof}

\begin{lem}
    \label{lem.H1Gminusn}
    Assume that $i_1=1$ and $p=r$. If $j_q=r-1$, then
    \[
    \wti{H}_{1}(\IN(G_n\setminus n))=0
    \quad \text{for all } n\ge 3r.
    \]
\end{lem}

\begin{proof}
    Observe that the edge ideal of $G_n\setminus n$ is the ideal of $R_{n-1}$ generated by monomials in $I_n$ that are not divisible by $x_n$, i.e. the ideal $\langle I_{n}\cap R_{n-1}\rangle_{R_{n-1}}$. Thus, if we define the chain $\Jcc=(J_n)_{n\ge 1}$ as follows
\[
J_n=\begin{cases}
0&\text{if}\quad n\le r-1\\
\langle I_{n+1}\cap R_{n}\rangle_{R_n} &\text{if}\quad n\ge r,
\end{cases}
\]
then $J_{n-1}$ is the edge ideal of $G_n\setminus n$ for $n\ge r+1$. By \cite[Lemma 4.7]{LN22}, $\Jcc$ is an $\Inc$-invariant chain with stability index $\ind(\Jcc)=\ind(\Icc)=r$. (Note that the ideal $J_n$ in the current proof is denoted by $J_{n+1}$ in \cite[Lemma 4.7]{LN22}, hence the difference in the stability indices.) A useful property of the chain $\Jcc$ is that the index $j_q$ for this chain increases by 1. Indeed, we have $x_{1}x_{j_q}\in I_r$ since $i_1=1$. It follows that $x_{1}x_{j_q+1}\in I_{r+1}$, and hence $x_{1}x_{j_q+1}\in J_{r}$, as claimed. Now the assumption that $j_q=r-1$ gives $x_{1}x_{r}\in J_r$.
So by \Cref{lem.reg1}, it holds that $\reg (R_{n-1}/J_{n-1})=1$ for all $n\ge 3r$ (in fact, it suffices to take $n\ge 3r-2$). This combined with \Cref{lem.FacetsofDegComplex}(i) and Takayama's formula yields
$$
\wti{H}_1(\IN(G_n\setminus n))
=\wti{H}_1(\Delta_\bfo(J_{n-1}))
\cong H^2_{\mm_{n-1}}(R_{n-1}/J_{n-1})_\bfo=0,
$$
where $\mm_{n-1}$ denotes the graded maximal ideal of $R_{n-1}.$
\end{proof}

\begin{lem}
\label{lem.homology.closedneighborhooddeletion}
Assume that $i_1=1$, $p=r$, $j_q< r$, and $\spi(\Icc)\ge 2$. Then for $n\ge 2r+1$, 
\[
\wti{H}_{0}(\IN(G_n\setminus N[n]))
\cong 
\begin{cases}
\kk &\text{if }\ j_q=r-1,\\
0 &\text{if }\ j_q\le r-2.
\end{cases}
\]
\end{lem}

\begin{proof}
Set $G\defas G_n\setminus N[n]$. Then for all $n\ge 2r+1$, we know from \Cref{lem.GWeaklyChordal} that $G$ has the vertex set $V(G)=V_1\cup V_2$, where 
\[
V_1=\{1,\ldots,b-1\}
\quad\text{and}\quad
V_2= \{n-r+B+1,\ldots,n-1\}.
\]
Denote by $\Gamma_1$ and $\Gamma_2$ the induced subgraphs of $G^c$ on $V_1$ and $V_2$, respectively.
We show that $\Gamma_1$ and $\Gamma_2$ are connected. Indeed, note that $G^c=G_n^c\setminus N[n]$ is the induced subgraph of $G_n^c$ on $V(G)=V_1\cup V_2$. Hence, $\Gamma_1$ and $\Gamma_2$ are also the induced subgraphs of $G_n^c$ on $V_1$ and $V_2$.
From the assumption $\spi(\Icc)\ge 2$ it follows that $(i,i+1)\in E(G_n^c)$ for all $i\in[n-1]$. Therefore, $\Gamma_1$ and $\Gamma_2$ are connected, as desired.

Let us first consider the case $j_q=r-1$, i.e. $(1,r-1)\in E(G_r)$. By \Cref{lem.vanishingofH0}, it suffices to show that $G^c$ has exactly two connected components. We claim that $\Gamma_1$ and $\Gamma_2$ are the connected components of $G^c$. 
Indeed, this means that $(u,v)\notin E(G^c)$, or equivalently, $(u,v)\in E(G)$ for every $u\in V_1$ and $v\in V_2$. As $n \ge 2r+1$, we have 
\[
v-u\ge (n-r+B+1)-(b-1)=n-r+B-b+2\ge r-2.
\]
Moreover, $v-(r-1)\le n-r$ since $v\le n-1$. It follows that
\[
0\le u-1 \le v-(r-1) \le n-r,
\]
which yields $(u,v)\in\Delta((1,r-1),n-r)$. Hence, $(u,v)\in E(G_n)$ by \Cref{lem.RightTriangle}. Since $G$ is the induced subgraph of $G_n$ on $V(G)$, this implies that $(u,v)\in E(G)$, as claimed.

Now assume that $j_q\le r-2$. Again by \Cref{lem.vanishingofH0}, 
we need to show in this case that $G^c$ is a connected graph. Since $\Gamma_1$ and $\Gamma_2$ are connected, it suffices to find an edge of $G^c$ connecting them. We claim that $(1,n-1)$ is such an edge. Suppose to the contrary that $(1,n-1)\notin E(G^c)$, i.e. $(1,n-1)\in E(G)$. This means that $(1,n-1)\in E(G_n)$ since $G$ is an induced subgraph of $G_n$. Thus, $(1,n-1)\in \Delta((i,j),n-r)$ for some $(i,j)\in E(G_r)$ by \Cref{lem.RightTriangle}. It follows that
\[
0\le 1-i \le n-1-j\le n-r.
\]
Hence, $i=1$ and $j\ge r-1$. But this contradicts the assumption that $j_q \le r-2$. Therefore, $(1,n-1)\in E(G^c)$, as desired.
\end{proof}

\begin{lem}
\label{lem.simple.H1GEqual0}
Assume that $i_1=1$, $p=r$, and $\spi(\Icc)\ge 2$. Then
$$
\wti{H}_{1}(\IN(G_n\setminus N[n]))=0
\quad \text{for all }
n\ge 3r.
$$
\end{lem}

The proof of this lemma is rather technical and lengthy. So we postpone it to the Appendix. Let us present here the proof of \Cref{prop.H1nonvanishing}.

\begin{proof}[Proof of \Cref{prop.H1nonvanishing}]
Let $n\ge 3r.$ The long exact sequence in \Cref{lem.MayerVietoris}, applied to the graph $G_n$ and its vertex $n$, together with the fact that $\wti{H}_{0}(\IN(G_n\setminus n))=0$ (\Cref{lem.H0Gminusn}) and $\wti{H}_{1}(\IN(G_n\setminus N[n])) = 0$ (\Cref{lem.simple.H1GEqual0}), yields the following short exact sequence
\begin{equation}
 \label{eq.ses}
 0 \longrightarrow \wti{H}_{1}(\IN(G_n\setminus n)) \longrightarrow \wti{H}_{1}(\IN(G_n)) \longrightarrow \wti{H}_{0}(\IN(G_n\setminus N[n])) \longrightarrow 0.    
 \end{equation}
Let us show that $\wti{H}_{1}(\IN(G_n))\cong \kk$ by induction on $r-j_q\ge 1$. If $r-j_q=1$, then \Cref{lem.H1Gminusn} gives $\wti{H}_1(\IN(G_n\setminus n))=0$. It thus follows from \eqref{eq.ses} and \Cref{lem.homology.closedneighborhooddeletion} that
\[
\wti{H}_1(\IN(G_n))
\cong\wti{H}_{0}(\IN(G_n\setminus N[n]))
\cong \kk.
\]

Now assume that $r-j_q>1$. In this case, $\wti{H}_{0}(\IN(G_n\setminus N[n]))=0$ by \Cref{lem.homology.closedneighborhooddeletion}. The short exact sequence \eqref{eq.ses} then implies that $\wti{H}_{1}(\IN(G_n))\cong \wti{H}_{1}(\IN(G_n\setminus n)).$
So it remains to show that $\wti{H}_{1}(\IN(G_n\setminus n)) \cong \kk.$
Consider the chain $\Jcc=(J_n)_{n\ge 1}$ as in the proof of \Cref{lem.H1Gminusn}. It is easy to verify that this chain 
satisfies all the assumptions of \Cref{prop.H1nonvanishing}.
Moreover, we know from the proof of \Cref{lem.H1Gminusn} that $x_{1}x_{j_q+1}\in J_{r}$. Thus we may apply the induction hypothesis to the chain $\Jcc$ and obtain $\wti{H}_{1}(\IN(G_n\setminus n)) \cong \kk$. 
This concludes the proof.
\end{proof}

The following example, which is somewhat similar to \Cref{ex.depthStabJq=r}, suggests that the lower bound for the index of stability of $\wti{H}_{1}(\IN(G_n))$ in \Cref{prop.H1nonvanishing} could be close to optimal.

\begin{ex}
Consider the chain $\Icc=(I_n)_{n\ge 1}$ with stability index $r\ge 6$ and 
\[
E(G_r)=\{(1,3), (2,r), (r-2,r)\},
\]
which satisfies all the assumptions of \Cref{prop.H1nonvanishing}.
Computations with Macaulay2 \cite{GS96} suggest that
\begin{align*}
\dim_\kk \wti{H}_1(\IN(G_n)) &= \begin{cases}
                                  0, &\text{if $n\le 2r-5$},\\
                                  2, &\text{if $2r-4\le n\le 3r-10$},\\
                                  1, &\text{if $n\ge 3r-9$}.
                    \end{cases}                                 
\end{align*}
That is, $\wti{H}_{1}(\IN(G_n))$ could be stable from $n=3r-9$. 
\end{ex}

We are now ready to prove \Cref{thm.limdepth.minimal}.

\begin{proof}[Proof of \Cref{thm.limdepth.minimal}]  
In view of \Cref{cor.limdepth.sple2}, it suffices to consider the case $\spi(\Icc)\ge 3$.
Moreover, using \Cref{lem.normalizingindexshift} we may furthermore assume that $i_1=1$, $p=r$ and thus reduce the statement we want to prove to 
\[
 \depth (R_n/I_n) =2
 \quad \text{for all }
n\ge 3r.
 \]
Let $n\ge 3r$. Combining \Cref{prop.H1nonvanishing} with \Cref{lem.FacetsofDegComplex}(i) and Takayama's formula we get
$$
\kk\cong \wti{H}_{1}(\IN(G_n)) 
=\wti{H}_1(\Delta_\bfo(I_{n}))
\cong H^2_\mm(R_n/I_n)_\bfo.
$$
It follows that $\depth(R_n/I_n) \le 2$. Consequently, $\depth(R_n/I_n) = 2$ by \Cref{thm_lowerbound_limdepth}. The proof is complete.
\end{proof}

The lower bound for the index of depth stability in \Cref{thm.limdepth.minimal} is also close to optimal, as illustrated by the next example.

\begin{ex}
 Consider the chain $\Icc=(I_n)_{n\ge 1}$ with stability index $r\ge 6$ and 
\[
E(G_r)=\{(1,r-1), (2,3), (2,r)\},
\]
which satisfies all the assumptions of \Cref{thm.limdepth.minimal}.
We show that $\depth(R_n/I_n)\ge 2$ for $n=3r-8$ and $\depth(R_n/I_n)=1$ for $n\ge 3r-7$.

\begin{figure}[ht]
\includegraphics[width=40ex]{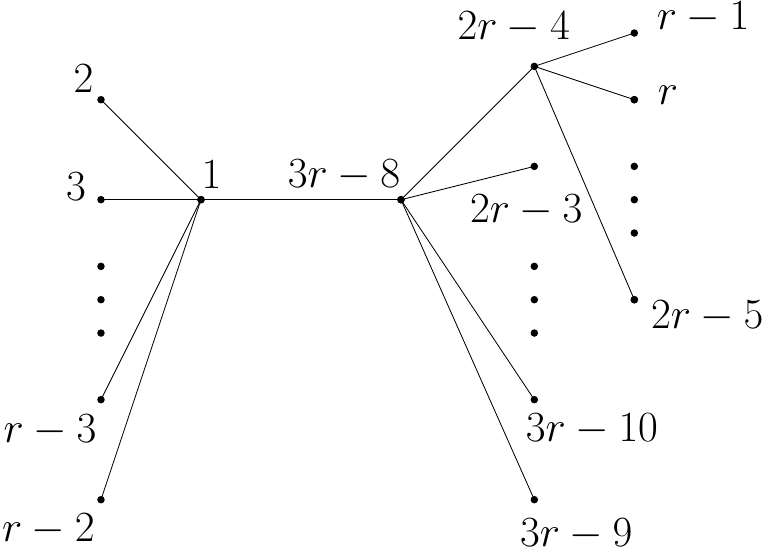}\\
\caption{A spanning tree of $G_{3r-8}^c$}
\label{fig.G(3r-8)c}
\end{figure}

To show $\depth(R_{3r-8}/I_{3r-8})\ge 2$, thanks to \Cref{cor_depthge2crit}, is equivalent to proving that $G_{3r-8}^c$ is connected. This holds because $G_{3r-8}^c$ has a spanning tree as in \Cref{fig.G(3r-8)c}. 

Next, to show $\depth(R_n/I_n)=1$ for $n\ge 3r-7$, we observe that $G_n^c$ is not connected for all such $n$, as it admits $r-1$ as an isolated vertex. Therefore, for each $r\ge 6$, the index of depth stability of the chain $\Icc$ is $3r-7$.
\end{ex}

\begin{rem}
\label{rem.minDepth}
As a summary of \Cref{thm.limdepth.maximal,thm.limdepth.minimal} and \cite[Theorem 7.1]{HNT2024}, \Cref{tab:multicol} provides a complete picture of the asymptotic depth and regularity of $\Inc$-invariant chains of (eventually non-zero) edge ideals. It would be interesting to have similar tables for more general chains, e.g. chains of (squarefree) monomial ideals.

\begin{table}[ht]
\caption{Asymptotic depth and regularity of $\Inc$-invariant chains of edge ideals}
\begin{center}
\begin{tabular}{|c|c|c|}
\hline
Invariant & Value & Condition\\
\hline
\multirow{4}{*}{$\lim\limits_{n\to \infty} \depth(R_n/I_n) $ }&\multirow{2}{*}{$r-\wti{r}+\spi(\Icc)$}& \multirow{2}{*}{either $j_q=p,$  or $\spi(\Icc)\le 2$ }\\ &&\\
    \cline{2-3}
    &\multirow{2}{*}{$r-\wti{r}+2$}&\multirow{2}{*}{$j_q<p$ and $\spi(\Icc)\ge 2$}  \\&&\\
\hline
 \multirow{4}{*}{$\lim\limits_{n\to \infty} \reg (R_n/I_n) $ }&\multirow{2}{*}{1}&\multirow{2}{*}{either $j_q=p,$ or ($\spi(\Icc)=1$ and $\inm(G_{3r})=1$)}   \\
 &&\\
\cline{2-3}
&\multirow{2}{*}{2}&\multirow{2}{*}{$j_q<p$ and  ($\spi(\Icc)\ge 2$ or  $\inm(G_{3r})=2$)} \\
 &&  \\
\hline
\end{tabular}
\end{center}
\label{tab:multicol}
\end{table}
\end{rem}

\section{Asymptotic homology of independence complexes}
\label{sec.homology}

In view of \Cref{prop.H1nonvanishing}, one may wonder whether all reduced homology groups of the independence complex $\IN(G_n)$ can be determined asymptotically. The main result of the present section provides a complete answer to this question. As before, we keep using \Cref{notn_chainofedgeids} throughout the section.

\begin{thm}
\label{thm_limhomology_main}
Assume that $i_1=1$ and $p=r$. Then there exist two nonnegative integers $\alpha, \beta$ depending only on the chain $\Icc$ such that the following hold for all $n\gg0$:
\begin{enumerate}
    \item 
    $\wti{H}_i(\IN(G_n))=0$ for $i\ne 0,1$.
    \item 
    $\wti{H}_0(\IN(G_n))
    \cong
    \begin{cases}
     \kk^{n-\alpha}&\text{if }\ \spi(\Icc)=1,\\
     0&\text{if }\ \spi(\Icc)\ge2.
    \end{cases}$
     \item 
     $\wti{H}_1(\IN(G_n))\cong\kk^\beta.$ Furthermore, if $\spi(\Icc)\ge2$, then 
$\beta=\begin{cases}
0&\text{if }\ j_q=r,\\
1&\text{if }\ j_q<r.
\end{cases}$
\end{enumerate}
\end{thm}

Notice that only the case $i_1=1$ and $p=r$ considered in the preceding theorem is nontrivial, because otherwise, $G_n$ contains isolated vertices for all $n\ge r$, which means that $\IN(G_n)$ is a cone and thus $\wti{H}_i(\IN(G_n))=0$ for every $i$. It is also worth noting that all reduced homology groups of $\IN(G_n)$, except for $\wti{H}_0(\IN(G_n))$ in the case $\spi(\Icc)=1$, are of bounded dimension.

To prove \Cref{thm_limhomology_main}, let us begin with its first statement.

\begin{lem}
\label{lem.Hge2}
For all $n\ge 4r$ and all $i\ne 0,1$, it holds that
$
\wti{H}_i(\IN(G_n))=0.
$
\end{lem}

\begin{proof}
Obviously, $\IN(G_n)\neq \{\emptyset\}$ for all $n\ge r$. Hence, $\wti{H}_i(\IN(G_n))=0$ for all $i<0$. Recall from \cite[Theorem 6.1]{HNT2024} that $\reg(R_n/I_n)\le 2$ for all $n\ge 4r$. Thus combining \Cref{lem.FacetsofDegComplex}(i) and Takayama's formula we get
\[
\wti{H}_i(\IN(G_n))=\wti{H}_i(\Delta_\bfo(I_{n}))\cong H^{i+1}_{\mm}(R_n/I_n)_\bfo=0
\]
for all $n\ge 4r$ and all $i\ge 2$. The desired conclusion follows.
\end{proof}

For the proof of the remaining part of \Cref{thm_limhomology_main}, note that the case $\spi(\Icc)\ge2$ has essentially been treated in \Cref{prop.H1nonvanishing} (for $\wti{H}_1$) and in the proof of \Cref{thm_lowerbound_limdepth} (for $\wti{H}_0$). To deal with the case  $\spi(\Icc)=1$, we need some more auxiliary results. In what follows, when $n\ge 4r$, we denote
\begin{equation}
    \label{eq.Un}
    U_n=\{2r, 2r+1,\ldots, n-2r\}.
\end{equation}

\begin{lem}
\label{lem.isolatedvertices}
Assume that $i_1=1,p=r$ and $\spi(\Icc)=1$. Then $U_n$ consists of isolated vertices of $G_n^c$ for all $n\ge 4r$.
\end{lem}

\begin{proof}
It follows from the assumption that $E(G_r)$ contains edges of the forms $(1,a),$ $(v-1,v),$ $(b,r)$, where $2 \le a, v\le r$ and $1\le b\le r-1$. Take any $k\in U_n$. We need to prove that $\{k,l\}\in E(G_n)$ for every $l\in [n]\setminus \{k\}$. By \Cref{lem.RightTriangle}, it suffices to show that either $(k,l)$ or $(l,k)$ belongs to one of the triangles $\Delta((1,a),n-r)$, $\Delta((v-1,v),n-r)$ and $\Delta((b,r),n-r).$ We distinguish two cases.

\emph{Case 1}: $l<k.$ In this case, one can easily check that
\[
(l,k)\in
\begin{cases}
    \Delta((1,a),n-r)&\text{if }\ 1\le l\le k-a+1,\\
    \Delta((v-1,v),n-r)&\text{if }\  k-a+2\le l\le k-1.
\end{cases}
\]

\emph{Case 2}: $l>k.$ It is readily seen that
\[
(k,l)\in
\begin{cases}
    \Delta((v-1,v),n-r)&\text{if }\ k+1\le l\le n-r+v,\\
    \Delta((b,r),n-r)&\text{if }\  n-r+v< l\le n. 
\end{cases}
\]
Thus we always have $\{k,l\}\in E(G_n)$, as desired.
\end{proof}

The next result determines the asymptotic dimension of $\wti{H}_0(\IN(G_n))$ when $\spi(\Icc)=1$. As before, the number of connected components of a graph $G$ is denoted by $\mathfrak{c}(G)$ . 

\begin{lem}
\label{lem.graphiso}
Assume that $i_1=1,p=r$ and $\spi(\Icc)=1$. For $n\ge 4r$, let $\Gamma_n=G_n^c\setminus U_n$ be the induced subgraph of $G_n^c$ on the vertex set
$$
V(\Gamma_n)=[n]\setminus U_n
=\{1,\dots,2r-1\}\cup\{n-2r+1,\dots,n\}.
$$
Then the following hold for all $n\gg 0$.
\begin{enumerate}
\item 
$\mathfrak{c}(G_n^c)=\mathfrak{c}(\Gamma_n)+n-4r+1.$ 
\item 
The graphs $\Gamma_n$ and $\Gamma_{n+1}$ are isomorphic. In particular, $\mathfrak{c}(\Gamma_{n+1})=\mathfrak{c}(\Gamma_n)$ and one can define $\mathfrak{c}(\Icc)\defas\mathfrak{c}(\Gamma_n)$ for $n\gg0$.
\item
$\dim_\kk \wti{H}_0(\IN(G_n))=n-\alpha$, where $\alpha=4r-\mathfrak{c}(\Icc)\ge 1$.
\end{enumerate}
\end{lem}

\begin{proof}
(i) Since $U_n$ consists of isolated vertices of $G_n^c$ by \Cref{lem.isolatedvertices}, we have 
\[
\mathfrak{c}(G_n^c)
=\mathfrak{c}(\Gamma_n)+|U_n|
=\mathfrak{c}(\Gamma_n)+n-4r+1.
\]

(ii) It suffices to prove the first assertion.
Consider the map $\sigma_{2r}$ as defined in \eqref{eq.sigma}.
Recall that $\sigma_{2r}\in\Inc_{n,n+1}$ for all $n\ge1$.
Let $\phi_n: V(\Gamma_n)\to V(\Gamma_{n+1})$ be the restriction of $\sigma_{2r}$ on $V(\Gamma_n)$. We show that $\phi_n$ is a graph isomorphism between $\Gamma_n$ and $\Gamma_{n+1}$ for all $n\gg 0$.
Evidently, $\phi_n$ is a bijection between $V(\Gamma_n)$ and $V(\Gamma_{n+1})$.
Denote by $\psi_n: V(\Gamma_{n+1})\to V(\Gamma_{n})$ the inverse map of $\phi_n$. We claim that $\psi_n(E(\Gamma_{n+1}))\subseteq E(\Gamma_{n})$. Equivalently, we need to show that if $\{i,j\}\notin E(\Gamma_{n})$, then $\{\phi_n(i),\phi_n(j)\}\notin E(\Gamma_{n+1})$. Indeed, if $\{i,j\}\notin E(\Gamma_{n})$, then $\{i,j\}\in E(G_{n})$, i.e. $x_ix_j\in I_n$. Since $\sigma_{2r} \in \Inc_{n,n+1}$, this implies that 
\[
x_{\phi_n(i)}x_{\phi_n(j)}=x_{\sigma_{2r}(i)}x_{\sigma_{2r}(j)}=\sigma_{2r}(x_ix_j) \in I_{n+1}.
\]
Hence, $\{\phi_n(i),\phi_n(j)\}\in E(G_{n+1})$, and thus $\{\phi_n(i),\phi_n(j)\}\notin E(\Gamma_{n+1})$, as desired.
So $\psi_n$ is a graph morphism from $\Gamma_{n+1}$ to $\Gamma_n$, which is a bijection on the vertex sets. Consequently, it yields an injective map $E(\Gamma_{n+1})\to E(\Gamma_n)$. In particular, $|E(\Gamma_n)|\ge |E(\Gamma_{n+1})|$ for all $n\ge 4r$. It follows that $|E(\Gamma_n)|= |E(\Gamma_{n+1})|$ for $n\gg 0$. In other words, for all $n\gg 0$, $\psi_n$ is a graph isomorphism, and hence so is its inverse $\phi_n$. 

(iii) The formula for the dimension of $\wti{H}_0(\IN(G_n))$ results from combining (i) and \Cref{lem.vanishingofH0}. We have $\alpha\ge 1$ since $\mathfrak{c}(\Icc)\le |V(\Gamma_n)|=4r-1.$
\end{proof}

It remains to determine the asymptotic dimension of $\wti{H}_1(\IN(G_n))$ when $\spi(\Icc)=1$. Given \Cref{lem.Hge2,lem.graphiso}, this can be done by using the \emph{Euler characteristic} of $\IN(G_n)$:
\[
\chi(\IN(G_n))
=\sum\limits_{i=0}^{d_n}(-1)^if_i(\IN(G_n))
=1+\sum_{i=-1}^{d_n} (-1)^i\dim_\kk \wti{H}_i(\IN(G_n)),
\]
where $d_n=\dim \IN(G_n)$ and $(f_i(\IN(G_n)))_{i=0}^{d_n}$ is the $f$-vector of $\IN(G_n)$. The following result describes the asymptotic behavior of the $f$-vector and the Euler characteristic of $\IN(G_n)$.

\begin{prop}
\label{prop.Hilbertseries.sp1}
Assume that $i_1=1$, $p=r$ and $\spi(\Icc)=1$. Then the following statements hold for all $n\gg 0$.
\begin{enumerate}
    \item 
    $f_0(\IN(G_{n+1}))=f_0(\IN(G_{n}))+1$ and $f_i(\IN(G_{n+1}))=f_i(\IN(G_{n}))$ for all $i\ge 1$.
    \item 
    $\chi(\IN(G_{n+1}))=\chi(\IN(G_n))+1$.
\end{enumerate}
\end{prop}

 \begin{proof}
 It is apparent that (ii) follows from (i), so we only need to prove (i). Observe that if $v$ is an isolated vertex of $G_n^c$, then it is also an isolated vertex of $\IN(G_n)$ and this gives
 \begin{equation}
 \label{eq.fvector}
    f_0(\IN(G_{n}))=f_0(\IN(G_{n}\setminus v))+1
    \quad\text{and}\quad
    f_i(\IN(G_{n}))=f_i(\IN(G_{n}\setminus v))
    \ \text{ for all } i\ge 1.
 \end{equation}
 For $n\ge 4r$ consider the set $U_n$ given in \eqref{eq.Un}. Recall from \Cref{lem.isolatedvertices} that $U_n$ consists of isolated vertices of $G_n^c$. Hence, applying \eqref{eq.fvector} repeatedly we obtain
 \begin{equation}
 \label{eq.fvector2}
     \begin{aligned}
     f_0(\IN(G_{n}))&=f_0(\IN(G_{n}\setminus U_n))+|U_n|
     =f_0(\IN(G_{n}\setminus U_n))+n-4r+1,\\
     f_i(\IN(G_{n}))&=f_i(\IN(G_{n}\setminus U_n))
     \ \text{ for all } i\ge 1.
     \end{aligned}
 \end{equation}
 Note that $G_{n}\setminus U_n=\Gamma_n^c$, where $\Gamma_n=G_{n}^c\setminus U_n$. By \Cref{lem.graphiso}, the graphs $\Gamma_n$ and $\Gamma_{n+1}$ are isomorphic for $n\gg 0$. This implies that the graphs $G_{n}\setminus U_n$ and $G_{n+1}\setminus U_{n+1}$ are isomorphic for $n\gg 0$. Consequently, the simplicial complexes $\IN(G_{n}\setminus U_n)$ and $\IN(G_{n+1}\setminus U_{n+1})$ are also isomorphic for $n\gg 0$. The desired conclusion now follows readily from \eqref{eq.fvector2}.
\end{proof}

\begin{rem}
    \Cref{prop.Hilbertseries.sp1}(i) implies that $\dim \IN(G_{n})$ is a constant for $n\gg0$. This also follows from \cite[Theorem 3.8]{LNNR1} and the fact that $I_n$ is the Stanley--Reisner ideal of $\IN(G_{n})$.
\end{rem}

We are now in a position to prove \Cref{thm_limhomology_main}.

\begin{proof}[Proof of \Cref{thm_limhomology_main}]
    The first statement follows from \Cref{lem.Hge2}. To prove the second and third statements, we distinguish two cases.

    \emph{Case 1}: $\spi(\Icc)\ge 2$. From the proof of \Cref{thm_lowerbound_limdepth} we know that the graph $G_n^c$ is connected for all $n\ge r$. Hence, $\wti{H}_0(\IN(G_n))=0$ for all $n\ge r$ by \Cref{lem.vanishingofH0}. Let us now prove the formula for $\wti{H}_1(\IN(G_n))$. In view of \Cref{prop.H1nonvanishing}, we only need to consider the case $j_q=r$. In this case, $\reg(R_n/I_n)=1$ for all $n\ge 3r-3$ by \Cref{lem.reg1}. So using \Cref{lem.FacetsofDegComplex}(i) and Takayama's formula, we get
    \[
    \wti{H}_{1}(\IN(G_n)) 
    =\wti{H}_1(\Delta_\bfo(I_{n}))
    \cong H^2_\mm(R_n/I_n)_\bfo
    =0
    \quad\text{for all } n\ge 3r-3.
    \]

    \emph{Case 2}: $\spi(\Icc)=1$. According to \Cref{lem.graphiso}(iii), there exists a constant $\alpha\ge1$ such that $\dim_\kk\wti{H}_0(\IN(G_n))=n-\alpha$ for $n\gg 0$. It remains to prove that $\dim_\kk\wti{H}_1(\IN(G_n))=\beta$ for some constant $\beta$ when $n\gg0$. Indeed, it follows from \Cref{prop.Hilbertseries.sp1}(ii) that there exists a constant $\gamma$ such that
    $$
   {\chi}(\IN(G_n))=n-\gamma
   \quad\text{for } n\gg0.
   $$ 
   Note that ${\chi}(\IN(G_n))=1+\dim_\kk\wti{H}_0(\IN(G_n))-\dim_\kk\wti{H}_1(\IN(G_n))$ for $n\gg0$ by \Cref{lem.Hge2}. Therefore,
   $$
   n-\gamma=1+(n-\alpha)-\dim_\kk\wti{H}_1(\IN(G_n)),
   $$
   and hence $\dim_\kk\wti{H}_1(\IN(G_n))=\gamma-\alpha+1$ for $n\gg0$, as desired.
\end{proof}

\Cref{prop.H1nonvanishing} and the proof of \Cref{thm_limhomology_main} provide lower bounds for the stability indices of $\wti{H}_0(\IN(G_n))$ and $\wti{H}_1(\IN(G_n))$ when $\spi(\Icc)\ge 2$, namely, $r$ and $3r$, respectively. 
It would therefore be interesting to have similar bounds in the case $\spi(\Icc)=1$. In this case, it would also be interesting to determine the constants $\alpha$ and $\beta$ in \Cref{thm_limhomology_main} explicitly. While $\alpha$ is always positive by \Cref{lem.graphiso}, the following examples indicate that $\beta$ can be zero or not.

\begin{ex}
Let $r=2$ and $E(G_2)=\{(1,2)\}$. Then $G_n$ is the complete graph $K_n$ for all $n\ge 2$. Thus, $\wti{H}_0(\IN(G_n)) \cong \kk^{n-1}$ and $\wti{H}_1(\IN(G_n)) =0$ for all $n\ge 2$.
\end{ex}

\begin{ex}
Let $r=4$ and $E(G_4)=\{(1,2),(3,4)\}$. We claim  that for all $n\ge 5$,
\[
 \wti{H}_0(\IN(G_n)) \cong \kk^{n-4} \quad \text{and} \quad \wti{H}_1(\IN(G_n)) \cong \kk.
\]
In fact, it is not hard to show that for all $n\ge 5$, the facets of $\IN(G_n)$ are precisely
\[
\{1,n-1\},\{2,n-1\},\{1,n\},\{2,n\},\{3\},\{4\},\ldots,\{n-2\}.
\]
The desired conclusion follows.
\end{ex}


\section{Appendix}

Here, as promised, we provide the proof of \Cref{lem.simple.H1GEqual0}. 
It is more convenient to prove the following slightly stronger result, which specializes to \Cref{lem.simple.H1GEqual0} when $a=b$ and $A=B$.

\begin{prop}
    \label{lem.H1GEqual0}
Assume that $i_1=1$, $p=r$ and $\spi(\Icc)\ge 2$. Let $a$ and $A$ be integers with $1\le a\le b$ and $B\le A \le r-1$. Denote by $G_n(a,A)$ the induced subgraph of $G_n$ on the vertex set $V(G_n(a,A))=V_1\cup V_2$, where $V_1=\{1,\dots,a-1\}$ and $V_2=\{n-r+A+1,\dots,n-1\}.$
Then 
$$
\wti{H}_{1}(\IN(G_n(a,A)))=0
\quad\text{for all }
n\ge 3r.
$$
\end{prop}

The proof of this result is mainly based on \Cref{prop.pdim.weaklychordal}. In order to apply \Cref{prop.pdim.weaklychordal}, some preparations are needed.

\begin{lem}
\label{lem.Gn.weakly.chordal}
Under the assumptions of \Cref{lem.H1GEqual0}, the following statements hold.
\begin{enumerate}
\item 
$G_n(a,A)$ is weakly chordal for all $n\ge 2r+1$.
\item 
$\inm(G_n(a,A))\le 2$ for all $n\ge 3r$.
\item Let $n\ge 3r$ and $(u_1,v_1),(u_2,v_2)\in E(G_n(a,A))$. If $(u_1,v_1),(u_2,v_2)$ with $u_1<u_2$ form an induced matching of $G_n(a,A)$, then 
\begin{align*}
  &1\le u_1<v_1\le a-1, \\
 &n-r+A+1\le u_2 < v_2 \le n-1,\\
 &a\ge 3 \quad \text{and} \quad r-A\ge 3.
\end{align*}
\end{enumerate} 
\end{lem}

\begin{proof}
(i) By \Cref{lem.GWeaklyChordal}, $G_n(b,B)=G_n\setminus N[n]$ is weakly chordal for all $n\ge 2r+1$. Since $G_n(a,A)$ is an induced subgraph of $G_n(b,B)$, we deduce that $G_n(a,A)$ is also weakly chordal.

(ii) The graph $G_n(a,A)$ is an induced subgraph of $G_n$. So by \cite[Theorem 3.1]{HNT2024},
\[
\inm(G_n(a,A)) \le \inm(G_n)\le 2
\quad\text{for all }
n\ge 3r.
\]

(iii) Assume that $(u_i,v_i) \in \Delta((k_i,l_i),n-r)$, where $(k_i,l_i)\in E(G_r)$ for $i=1,2$. Since $G_n(a,A)$ is an induced subgraph of $G_n$, $\{(u_1,v_1),(u_2,v_2)\}$ is also an induced matching of $G_n$. From the proof of \Cref{lem.K2} we know that 
$$
v_1 <k_2 < n-r+l_1 <u_2.
$$ 
As $V(G_n(a,A)) = \{1,\ldots ,a-1 \} \cup \{n-r+A+1,\ldots , n-1\}$ and $\max\{a-1,k_2\}<r<n-r$,
it follows that
\begin{equation*}
1\le u_1<v_1\le a-1 \text{ and  } n-r+A+1\le u_2 < v_2 \le n-1.
\end{equation*}
In particular, these inequalities yield $a\ge 3$ and $r-A\ge 3$.
\end{proof}

\begin{lem}
    \label{lem.consecutive}
    Keep the assumptions of \Cref{lem.H1GEqual0}. Assume further that $\mathfrak{B}$ is a complete bipartite subgraph of $G_n(a,A)$ with partition $V(\mathfrak{B})=W_1\cup W_2$, where $W_1,W_2\ne\emptyset$. If $U\subseteq V(\mathfrak{B})$ consists of consecutive integers, then either $U\subseteq W_1$ or $U\subseteq W_2$.
\end{lem}

\begin{proof}
It suffices to prove that $U\subseteq W_1$ if $U\cap W_1\ne\emptyset$. Indeed, take $k\in U\cap W_1$. Let $l=k-1$ or $l=k+1$.
Then $\{k,l\}\notin E(\mathfrak{B})$ since $\spi(\Icc)\ge 2$. Hence, $l\in W_1$ whenever $l\in V(\mathfrak{B})$. Since $U$ consists of consecutive integers, it follows easily by induction that $U\subseteq W_1$.
\end{proof}

\begin{lem}
\label{lem.Gn.no.contains.twocompletebipartite}
Keep the assumptions of \Cref{lem.H1GEqual0}. Then for all $n\ge 3r$, $G_n(a,A)$ does not contain a strongly disjoint family of two complete bipartite subgraphs $\mathfrak{B}_1,\mathfrak{B}_2$ such that $V(\mathfrak{B}_1)\cup V(\mathfrak{B}_2) = V(G_n(a,A))$.
\end{lem}

\begin{proof}
Assume the contrary that there exists a strongly disjoint family of two complete bipartite subgraphs $\mathfrak{B}_1,\mathfrak{B}_2$ of $G_n(a,A)$ with $V(\mathfrak{B}_1)\cup V(\mathfrak{B}_2) = V(G_n(a,A))$ for some $n\ge 3r$.
Then $G_n(a,A)$ has an induced matching $(u_1,v_1),(u_2,v_2)$, where $(u_i,v_i)\in E(\mathfrak{B}_i)$ for $i=1,2$.
We may assume that $u_1<u_2$. Then \Cref{lem.Gn.weakly.chordal}(iii) yields
\begin{equation*}
1\le u_1<v_1\le a-1 \text{ and  } n-r+A+1\le u_2 < v_2 \le n-1.
\end{equation*}
Let $V(\mathfrak{B}_1) = W_{1}\cup W_{2}$ be the vertex partition of $\mathfrak{B}_1$. We first show that $k\notin V(\mathfrak{B}_1)$ for some $k< v_1$. In fact, if $[v_1] \subseteq V(\mathfrak{B}_1)$, then it follows from \Cref{lem.consecutive} that either $[v_1]\subseteq W_1$ or $[v_1]\subseteq W_2$. But this contradicts the fact that $(u_1,v_1) \in E(\mathfrak{B}_1)$. Hence, there must exist $k< v_1$ such that $k\notin V(\mathfrak{B}_1)$.

As $k<v_1\le a-1$, we have $k\in V(G_n(a,A))=V(\mathfrak{B}_1)\cup V(\mathfrak{B}_2)$, and thus $k\in V(\mathfrak{B}_2)$. Since $\mathfrak{B}_2$ is complete bipartite and $(u_2,v_2) \in E(\mathfrak{B}_2)$, we deduce that either $(k,u_2)$ or $(k,v_2)$ belongs to $E(\mathfrak{B}_2)\subseteq E(G_n)$. Consider the case $(k,u_2)\in E(G_n)$; the case $(k,v_2)\in E(G_n)$ being similar. Since $k<v_1 \le r$ and $ n-r < u_2$, moving east from $(k,u_2)$ to $(v_1,u_2)$ using \Cref{lem.shortmoves}(i), we get $(v_1,u_2) \in E(G_n)$. Consequently, $(v_1,u_2)\in E(G_n(a,A))$, as $G_n(a,A)$ is an induced subgraph of $G_n$. But this is impossible because $(u_1,v_1),(u_2,v_2)$ is an induced matching of $G_n(a,A)$. The desired conclusion follows.
\end{proof}

\begin{lem}
\label{lem.Gn.no.contains.onecompletebipartite}
Keep the assumptions of \Cref{lem.H1GEqual0}. Assume also that $\inm(G_n(a,A))=2$ and that $G_n(a,A)$ has a complete bipartite subgraph $\mathfrak{B}$ with $|V(\mathfrak{B})| \ge |V(G_n(a,A))|-1$ and $E(\mathfrak{B})\ne \emptyset$ for some $n\ge 3r$. If $\wti{H}_{1}(\IN(G_n(a-1,A)))=\wti{H}_{1}(\IN(G_n(a,A+1)))=0$, then $\wti{H}_{1}(\IN(G_n(a,A)))=0$.
\end{lem}

\begin{proof}
Let $\{(u_1,v_1),(u_2,v_2)\}$ with $u_1<u_2$ be an induced matching of $G_n(a,A)$.
Then we know from \Cref{lem.Gn.weakly.chordal}(iii) that
\begin{equation*}
1\le u_1<v_1\le a-1 \text{ and  } n-r+A+1\le u_2 < v_2 \le n-1.
\end{equation*}
Let $V(\mathfrak{B}) = W_{1}\cup W_{2}$ be the vertex partition of $\mathfrak{B}$. Then $W_1,W_2\ne\emptyset$ since $E(\mathfrak{B})\ne \emptyset$. Recall that $V(G_n(a,A))=V_1\cup V_2$, where both $V_1$ and $V_2$ consist of consecutive integers. Since $|V(\mathfrak{B})| \ge |V(G_n(a,A))|-1$, either $V_1$ or $V_2$ is contained in $V(\mathfrak{B})$. So by reindexing (if needed), it follows from \Cref{lem.consecutive} that either $V_1\subseteq W_2$ or $V_2\subseteq W_2$. We consider only the case $V_2\subseteq W_2$; the other case can be treated similarly. In this case, $u_2,v_2\in V_2\subseteq W_2$. Since $(u_1,v_1),(u_2,v_2)$ form an induced matching of $G_n(a,A)$, we infer that 
\begin{equation}
\label{eq.notinW1}
\text{neither $u_1$ nor $v_1$ belongs to $W_1$.}
\end{equation}
Hence, either $u_1$ or $v_1$ belongs to $W_2$ because $|V(\mathfrak{B})| \ge |V(G_n(a,A))|-1$. Thus $V_1\cap W_2\ne \emptyset$. This together with \Cref{lem.consecutive} and the fact that $W_1\ne\emptyset$ yields $V_1\nsubseteq V(\mathfrak{B})$. The assumption $|V(\mathfrak{B})| \ge |V(G_n(a,A))|-1$ now imples that
\[
V(\mathfrak{B})=(V_1\setminus\{k\})\cup V_2
\]
for some $k\in V_1$.
Note that $k\ge 2$, since otherwise, $V_1\setminus \{1\} \subseteq V(\mathfrak{B})$. So \Cref{lem.consecutive} yields $V_1\setminus \{1\}\subseteq W_1$ or $V_1\setminus \{1\}\subseteq W_2$. As $W_1\neq \emptyset$, we deduce that $V_1\setminus \{1\}\subseteq W_1$. But then either $u_1$ or $v_1$ belongs to $W_1$, contradicting \eqref{eq.notinW1}.

\begin{claim}
\label{cl.partition}
    It holds that 
    \[
    W_1=\{k+1,\dots,a-1\}
    \quad\text{and}\quad
    W_2=\{1,\ldots,k-1\} \cup V_2.
    \]
\end{claim}

Indeed, we have $1\in V(\mathfrak{B})$ because $k\ge 2$. If $1\in W_1$, then $(1,u_2)\in E(\mathfrak{B})\subseteq E(G_n)$ since $u_2\in W_2$. Recall that $1\le u_1 \le r$ and $n-r <u_2$. So according to \Cref{lem.shortmoves}(i), it holds that $(u_1,u_2)\in E(G_n)$, whence $(u_1,u_2)\in E(G_n(a,A))$. But this contradicts the fact that $\{(u_1,v_1),(u_2,v_2)\}$ is an induced matching of $G_n(a,A)$. Thus we must have $1\in W_2$. This together with \Cref{lem.consecutive} implies $[k-1]\subseteq W_2$. Hence, $[k-1]\cup V_2\subseteq W_2$. Now since $W_1\ne\emptyset$, the desired claim follows easily from \Cref{lem.consecutive}.

\begin{claim}
\label{cl.neighborhood}
    The closed neighborhood of $a-1$ in $G_n(a,A)$ is 
$$
N[a-1]=\{1,\ldots,a-3,a-1\} \cup V_2 
= V(G_n(a,A))\setminus \{a-2\}.
$$
\end{claim}

Since $\spi(\Icc)\ge 2$, we have $a-2\notin N[a-1]$, hence $N[a-1]\subseteq V(G_n(a,A))\setminus \{a-2\}$. On the other hand, \Cref{cl.partition} gives $W_2\subseteq N[a-1]$. Thus it remains to show that
\begin{equation}
   \label{eq.Na}
   \{k,k+1,\ldots,a-3\} \subseteq N(a-1).
\end{equation}

From \Cref{cl.partition} we know that $(k-1,k+1)\in E(\mathfrak{B})\subseteq E(G_n).$ By \Cref{lem.RightTriangle} and the assumption that $\spi(\Icc)\ge 2$, this implies that $(k-1,k+1)\in\Delta((l,l+2),n-r)$ for some $(l,l+2)\in E(G_r)$. In particular, one has $l\le k-1$. Now using \Cref{lem.RightTriangle}, it is easy to check that $(i,a-1)\in\Delta((l,l+2),n-r)$ for any $k\le i\le a-3$. Hence, \eqref{eq.Na} is true, as required.

\Cref{cl.neighborhood} implies that $G_n(a,A)\setminus N[a-1]\cong K_1$, where $K_1$ is the complete graph on one vertex. Moreover, it is clear 
that $G_n(a,A)\setminus \{a-1\}=G_n(a-1,A)$. So applying \Cref{lem.MayerVietoris} to $G_n(a,A)$ and its vertex $a-1$, and using the fact that $\wti{H}_0(\IN(K_1))=\wti{H}_1(\IN(K_1))=0$, we obtain 
\[
\wti{H}_1(\IN(G_n(a,A))) \cong \wti{H}_1(\IN(G_n(a-1,A)))=0.
\]
The proof is complete.
\end{proof}

We are now prepared to give the proof of \Cref{lem.H1GEqual0}.

\begin{proof}[Proof of \Cref{lem.H1GEqual0}]
Fix an $n\ge 3r$. We show that $\wti{H}_{1}(\IN(G_n(a,A)))=0$ by a double induction on $a \in [b]$ and on $r-A \in [r-B]$.
Let $S=\kk[x_i\mid i\in V(G_n(a,A))]$ and denote by $L(a,A)\subseteq S$ the edge ideal of $G_n(a,A)$.
Recall that the graph $G_n(a,A)$ is weakly chordal by \Cref{lem.Gn.weakly.chordal}(i). So \cite[Theorem 14]{Wo14} yields
\begin{equation*}
\reg(S/L(a,A)) = \inm(G_n(a,A)).
\end{equation*}

According to \Cref{lem.Gn.weakly.chordal}(ii), we have $\inm(G_n(a,A))\le 2$. Let us first consider the case $\inm(G_n(a,A))\le 1$. Notice that this case covers the case that either $a\le 2$ or $r-A\le 2$ by virtue of \Cref{lem.Gn.weakly.chordal}(iii). Since $\reg(S/L(a,A))= \inm(G_n(a,A)) \le 1$, it follows from Takayama's formula and \Cref{lem.FacetsofDegComplex}(i) that 
\[
\wti{H}_1(\IN(G_n(a,A))) \cong H^2_{\mm}(S/L(a,A))_\bfo =0,
\]
where $\mm$ denotes the graded maximal ideal of $S$.

Now assume that $\inm(G_n(a,A))=2.$ In this case, $a\ge 3$ and $r-A\ge 3$ by \Cref{lem.Gn.weakly.chordal}(iii). Since $G_n(a,A)$ is weakly chordal,
\Cref{prop.pdim.weaklychordal} implies that there exists a strongly disjoint family of complete bipartite subgraphs $\mathfrak{B}_1,\ldots, \mathfrak{B}_g$ of $G_n(a,A)$ with $1\le g\le \inm(G_n(a,A))$ such that
\begin{equation*}
\pd (S/L(a,A))=\sum_{i=1}^g|V(\mathfrak{B}_i)|-g.
\end{equation*}
Since $\inm(G_n(a,A))=2$, there are only two cases to consider.

\emph{Case 1}: $g=2$. By \Cref{lem.Gn.no.contains.twocompletebipartite}, $V(\mathfrak{B}_1)\cup V(\mathfrak{B}_2) \subsetneq V(G_n(a,A))$. It follows that
$$
\pd(S/L(a,A))=|V(\mathfrak{B}_1)|+|V(\mathfrak{B}_2)|-2\le |V(G_n(a,A))|-1-2= \dim(S)-3.
$$
Hence, $\depth(S/L(a,A))\ge 3$ by the Auslander--Buchsbaum formula. This together with Takayama's formula and \Cref{lem.FacetsofDegComplex}(i) implies that 
$$
\wti{H}_1(\IN(G_n(a,A))) \cong H^2_{\mm}(S/L(a,A))_\bfo =0.
$$

\emph{Case 2}: $g=1$. If $|V(\mathfrak{B}_1)|\le |V(G_n(a,A))|-2$, then again 
$$
\pd(S/L(a,A))=|V(\mathfrak{B}_1)|-1\le |V(G_n(a,A))|-2-1 = \dim(S)-3,
$$
and we reach the desired conclusion with the same argument as in the previous case. Now suppose that $|V(\mathfrak{B}_1)|\ge |V(G_n(a,A))|-1$. From the induction hypothesis we know that
$$\wti{H}_{1}(\IN(G_n(a-1,A)))=\wti{H}_{1}(\IN(G_n(a,A+1)))=0.$$ 
So by
\Cref{lem.Gn.no.contains.onecompletebipartite}, 
$\wti{H}_1(\IN(G_n(a,A)))=0$. This completes the proof.
\end{proof}

\begin{acknowledgment}
The first and fifth named authors (TQH and TTN) were supported by the Vietnam National Foundation for Science and Technology Development (NAFOSTED) under the grant number 101.04-2023.07. The fourth named author (HDN) was supported by NAFOSTED under the grant number 101.04-2023.30.  Part of this work was done while the authors were visiting the Vietnam Institute for Advanced Study in Mathematics (VIASM). We would like to thank VIASM for hospitality and financial support.
\end{acknowledgment}

\end{document}